\documentclass[12pt,reqno]{amsart}

\usepackage{amsmath}
\usepackage{amssymb}
\usepackage{amsbsy}
\usepackage{amscd}
\usepackage{amsfonts}
\usepackage{amsopn}
\usepackage{amstext}
\usepackage{amsthm}
\usepackage{amsxtra}
\usepackage{color}
\usepackage{enumerate}
\everymath={\displaystyle}

\newcommand{\beq}{\begin{equation}}
\newcommand{\eeq}{\end{equation}}

\newcommand{\bee}{\begin{eqnarray}}
\newcommand{\ene}{\end{eqnarray}}
\newcommand{\bea}{\begin{eqnarray*}}
\newcommand{\ena}{\end{eqnarray*}}

\theoremstyle{plain}
\newtheorem{theo}{Theorem}[section]
\newtheorem{lemma}[theo]{Lemma}
\newtheorem{coro}[theo]{Corollary}
\newtheorem{prop}[theo]{Proposition}

\newtheorem*{rem}{Remark}

\newtheorem*{Data Availability}{Data Availability}

\newtheorem*{hypothesis}{Hypothesis S}
\newtheorem{hypothesis1}[theo]{Conjecture }

\theoremstyle{definition}

\numberwithin{equation}{section}

\begin{document}

\title[]{\textbf {C\MakeLowercase{ancellation in sums over special sequences on} }$\mathbf{{\rm{GL}}_{m}}$ \textbf{\MakeLowercase{and their applications}}}

\author{Qiang Ma}
\address{Institute for Advanced Study in Mathematics, Zhejiang University, Hangzhou, Zhejiang 310000, People's Republic of China}
\email{qma829@zju.edu.cn}

\author{Rui Zhang}
\address{School of Mathematics, Shandong University, Jinan, Shandong 250100, People's Republic of China}
\email{rzhang@mail.sdu.edu.cn}

  \begin{abstract}
 Let $a(n)$ be the $n$-th Dirichlet coefficient
 of the automorphic $L$-function or the Rankin--Selberg $L$-function.   We investigate the cancellation  of  $a(n)$ over sequences linked to the Waring--Goldbach problem, by establishing a nontrivial bound  for the additive twisted sums over primes on ${\mathrm{GL}}_m .$ The bound does not depend on the generalized Ramanujan conjecture or the nonexistence of Landau--Siegel zeros. Furthermore, we present an application associated with the Sato--Tate conjecture and propose a conjecture about  the Goldbach conjecture on average bound. 
\end{abstract}
\maketitle
\pagestyle{empty}
 \section{introduction}
In the field of number theory, the additive and multiplicative structures of integers are two major areas of research.  Multiplicative number theory investigates the distribution of prime numbers through the analysis of generating functions, with the Prime Number Theorem and Dirichlet's theorem being the most famous classical achievements in this field. On the other hand, additive number theory explores the additive decomposition of integers, with the Waring--Goldbach problem being a well-known subject. It has been generally assumed that these two structures are independent, but no proof has confirmed this. Problems involving both structures tend to be more challenging to solve, such as the twin prime conjecture,  which asserts that there should be infinitely many pairs of primes which differ by $2$. Although this conjecture remains open, recent work by Zhang \cite{Zhang-2014},  Maynard \cite{May-2015}, Polymath \cite{Polymath-2014}, and many other scholars has made significant progress.  Merging both additive and multiplicative aspects within number theory, this paper specifically explores coefficients of the standard $L$-functions over sequences related to the Waring--Goldbach problem.

Let \(\mathcal{P}\) be the set of all primes, \(N\) be a sufficiently large positive integer.
  Define the sets
$$
\mathcal{P}_{k}(N):=\{p\in\mathcal{P}, p^k\le N\},\quad 
$$
and
\begin{align*}
\mathcal{J}_{k,u}(N):= &\{(p_1, p_2,\dots,p_u)|\ p_i\in \mathcal{P}_{k}(N) ~ (1\le i\le u ), N=p_1^k+p_2^k+\cdots+p_u^k\},
\end{align*}
where $k$ and $u$ are positive integers.   For a fixed \(k\), the objectives of the Waring--Goldbach problem are to determine the smallest \(u\) such that the set \(\mathcal{J}_{k,u}(N)\) is non-empty and to provide the asymptotic formula for the number of elements in \(\mathcal{J}_{k,u}(N)\). 
We introduce two special cases associated with \(\mathcal{J}_{k,u}(N)\): the linear case \(k=1\) and the quadratic case \(k=2\). In both cases, $N$ is assumed to be an odd integer throughout this paper.

For $k=1$, Vinogradov \cite{Vi37} proved  
 $$
\# \mathcal{J}_{1,3}(N)=\frac12(1+o(1))\mathfrak{G}_{1,3}(N)\frac{N^2}{(\log N)^{3}},
$$
where $\mathfrak{G}_{1,3}(N)$ is the singular series
\begin{equation}\label{singular series}
\mathfrak{G}_{1,3}(N)=\prod_{p\mid N}\left(1-\frac{1}{(p-1)^{2}}\right)\prod_{p\nmid N}\left(1+\frac{1}{(p-1)^{3}}\right).
 \end{equation}
The equation implies that every sufficiently large odd integer \(N\) can be expressed as the sum of three primes, representing Vinogradov's three primes theorem. Furthermore, \(u=3\) is the minimum number of variables required for \(\mathcal{J}_{1,u}(N)\) to be non-empty.  For the quadratic case $k=2$, and for all sufficiently large $N$ satisfying $N\equiv5\ ({\rm mod}\ 24)$,  Hua \cite{Hua-1938} gave the best result and showed that $\#\mathcal{J}_{2,5}(N)>0,$ where the number of variables required is still the best (for more details, see Lemma \ref{4}).

Inspired by the above result, many scholars have considered the analogue of Vinogradov's three primes theorem for the coefficients of the standard $L$-functions. Let $c$ be a positive number, whose value may vary throughout the paper.  Fouvry and Ganguly \cite{F-G-2014} proved that  there exists an effective constant $c>0$ such that
\begin{equation}\label{V-1}
\sum_{(p_1,p_2,p_3)\in\mathcal{J}_{1,3}(N)}\nu_f(p_1)\ll_f N^{2} \exp \left(-c\sqrt{\log N}\right),
\end{equation}
where $\nu_f (n)$ is the normalized Fourier coefficient of a primitive holomorphic or Maass cusp form $f$.
Let  $F$ be a Hecke--Maass form for $\mathrm{SL}(3, \mathbb{Z})$, with $A_F(n, 1)$ denoting the $n$th coefficient of its Dirichlet series.  
 Hou \emph{et al.} \cite{Hou-16} proved that there exists an effective constant $c>0$ such that
\begin{equation}\label{V-2}
\sum_{(p_1,p_2,p_3)\in\mathcal{J}_{1,3}(N)}A_F(p_1, 1)  \ll_F N^{2} \exp (-c \sqrt{\log N}).
\end{equation}
 For a primitive Dirichlet character $\chi$,   Hoffstein and Ramakrishnan \cite[Theorem A]{H-R-1995} and Banks \cite{B-1997} have proved  that  the twisted $L$-functions $L(s, f\otimes\chi)$ and $L(s, F \otimes \chi)$ don't exist Landau--Siegel  zeros.  The above two estimates depend on the nonexistence of Landau--Siegel  zeros. 
Since  the existence of Landau--Siegel zeros is still open for the $L$-functions with higher ranks, the method used in these results cannot be applied to similar problems for higher ranks. In this paper, we introduce a new approach to obtain related estimates for $L$-functions and Rankin--Selberg $L$-functions with higher ranks, similar to the  estimates \eqref{V-1} and \eqref{V-2}.
Moreover, our method can be applied to consider the cancellation of the coefficients associated with the quadratic case, which is more intricate. Inspired by the approach for handling the Waring--Goldbach problem,  we employ several different estimates of exponential sums to obtain the optimal upper bound (for more details, see Section \ref{pr-th-1.4}).  

Let $\mathbb{A}$ be  the ring of adeles over \(\mathbb{Q}\). Define a unitary Hecke character \(\omega: \mathbb{Q}^{\times}\backslash \mathbb{A}^{\times} \rightarrow \mathbb{C}^{\times}\). Consider the set
\[
\mathfrak{E}_m:=\{\pi \subset L^{2}(\mathrm{GL}_{m}(\mathbb{Q}) \backslash  \mathrm{GL}_{m}(\mathbb{A}); \omega) \mid \omega \text{ is trivial, }\pi \text{ is irreducible and cuspidal} \}.
\]
 Given \(\pi \in \mathfrak{E}_m\), let \(\mathfrak{q}_\pi\) be the conductor of \(\pi\), \(L(s, \pi)\) be the associated standard \(L\)-function, and \(\widetilde{\pi} \in \mathfrak{E}_m\) be the contragredient representation. We write $\lambda_\pi(n)$ to be the $n$-th Dirichlet coefficient of $L(s, \pi)$.  Let $\chi $ be a Hecke character of the idele class group $\mathbb{Q}^{\times}\backslash \mathbb{A}^{\times} $. We define the set 
 $$\begin{aligned}
 \mathfrak{E}_m^b:=\{\pi \in \mathfrak{E}_m&|\pi=\widetilde{\pi} \text{ and } \pi \neq \pi \otimes \chi \\& \text{ for  all nontrivial primitive quadratic Dirichlet characters }\chi\}.
 \end{aligned}$$ 
We will examine the cancellation  of  the multiplicative  Dirichlet coefficients $\lambda_{\pi}(n)$ 
over specific sequences $\mathcal{J}_{1,3}(N)$ and $\mathcal{J}_{2,5}(N)$, and provide some applications. 
\begin{theo}\label{<2>}
Let $\pi\in\mathfrak{E}_m^b$ and $B>0$ be a sufficiently large constant. Then 
$$
r_\pi(N):=\sum_{(p_1,p_2,p_3)\in\mathcal{J}_{1,3}(N)} \lambda_{\pi}\left(p_1\right) \ll_{\pi, B} N^{2}(\log N)^{-B},
$$
and
$$
w_\pi(N):=\sum_{(p_1,p_2,p_3,p_4,p_5)\in\mathcal{J}_{2,5}(N)} \lambda_{\pi}\left(p_1\right) \ll_{\pi, B} N^{3/2}(\log N)^{-B}.$$
\end{theo} 
Due to the potential existence of these zeros, the  bounds in the above theorem are weaker than those in \eqref{V-1} and \eqref{V-2}, and the condition $\pi\in\mathfrak{E}_{m}^{b}$ is required to obtain Theorem \ref{<2>}.  Assuming an enhanced zero-free region for $\pi \in \mathfrak{E}_{m}$,  sharper upper bounds can be derived for a broader class of $L$-functions.
\begin{hypothesis}\label{hy}
Let $1/2\le\delta<1$.  For  any $\pi \in   \mathfrak{E}_m$ with $m\ge 1$, the  $L$-functions $L(s,\pi )$  have no zeros for $\mathrm{Re}(s)>\delta$. 
\end{hypothesis}
\begin{theo}\label{<3>}
 Let $\pi\in\mathfrak{E}_m$ and $\varepsilon$ be a sufficiently small positive number.  Under Hypothesis S,  
$$
r_\pi(N) \ll_{\varepsilon} m^{2} N^{\frac{5}{3}+\frac{\delta}{3}+\varepsilon}(\log (N \mathfrak{q}_{\pi}))^3,
$$
and
$$
    w_\pi(N) \ll_{\varepsilon} m^{2} (N^{\frac{23}{16}+\varepsilon}+N^{\frac{4}{3}+\frac{\delta}{6}+\varepsilon})(\log( N\mathfrak{q}_{\pi}))^{3}.
    $$ 
    \end{theo}
\begin{rem}
\hspace*{\fill}
\begin{enumerate} 
\item 
Regarding the upper bounds in \eqref{V-1} and \eqref{V-2},  we can use the same method as in Theorem \ref{<3>} to obtain them. 
\item
 Note that the upper bound for $w_\pi(N)$  consists of two terms. This is due to the use of Dirichlet approximation theorem, which allows us to divide the minor arcs twice.  Such an approach is effective in handling minor arcs within the circle method.
\end{enumerate}
\end{rem}
 
In our above cases, we do not provide asymptotic formulas. But  the upper bounds we obtain are better than the error terms in the classical cases (see Lemma \ref{4}). Furthermore, we utilize the relationship between the Rankin--Selberg $L$-function and the adjoint $L$-function to obtain  asymptotic formulas concerning  the Rankin--Selberg coefficients over $\mathcal{J}_{k,u}(N)$. For $\pi\in\mathfrak{E}_m$, the construction of an adjoint $L$-function is described by employing a Rankin--Selberg $L$-function and zeta function:
\begin{equation}\label{jijiji1}
L(s, \operatorname{Ad}\pi):=\frac{L(s, \pi \times \tilde{\pi})}{\zeta(s)},
\end{equation}
where $\operatorname{Ad}\pi$ is the adjoint lift of $\pi$.
The adjoint $L$-function is meromorphic since $L(s, \pi \times \widetilde{\pi})$ and $\zeta(s)$ are meromorphic. Details regarding the Rankin--Selberg $L$-functions can be found in Section 2. For $\mathrm{Re}(s)>1$, we express
$$
L(s,  \operatorname{Ad}\pi)=\prod_p \underset{ (i, j) \neq (1,1)}{\prod_{i=1}^m\prod_{j=1}^{m}}\left(1-\frac{\alpha_{i, \pi}(p)\overline{\alpha_{j, \pi}(p)}}{p^s}\right)^{-1}=\sum_{n=1}^{\infty} \frac{\lambda_{\operatorname{Ad}\pi}(n)}{n^{s}}.
$$
By the relation between $\lambda_{\pi \times \widetilde{\pi}}(p)$ and $\lambda_{\operatorname{Ad}\pi}(p)$, we easily obtain the following results.
\begin{coro}\label{twist}
 Let $\pi\in\mathfrak{E}_m^b$,  $\varepsilon$ be a  sufficiently small positive number and $B>0$ be a sufficiently large constant.   \
\begin{enumerate}
\item Under the hypothesis that $\operatorname{Ad}\pi\in\mathfrak{E}_{m^{2}-1}^b$, 
 $$
r_{\pi\times\tilde{\pi}}(N):=\sum_{(p_1,p_2,p_3)\in\mathcal{J}_{1,3}(N)} \lambda_{\pi\times\tilde{\pi}}\left(p_1\right) =\# \mathcal{J}_{1,3}(N)+O_{\pi,B}(N^{2}(\log N)^{-B}),$$
and
$$
w_{\pi\times\tilde{\pi}}(N):=\sum_{(p_1,p_2,p_3,p_4,p_5)\in\mathcal{J}_{2,5}(N)} \lambda_{\pi\times\tilde{\pi}}\left(p_1\right) =\# \mathcal{J}_{2,5}(N)+O_{\pi,B}(N^{2}(\log N)^{-B}).
$$
\item Under Hypothesis S and  the hypothesis that $\operatorname{Ad}\pi\in\mathfrak{E}_{m^{2}-1}^b$, 
$$
        r_{\pi\times\tilde{\pi}}(N) =\# \mathcal{J}_{1,3}(N)+O_{\varepsilon}(m^{2}N^{\frac{5}{3}+\frac{\delta}{3}+\varepsilon}(\log ( N\mathfrak{q}_{\pi}))^3),$$
and
$$
w_{\pi\times\tilde{\pi}}(N) =\# \mathcal{J}_{2,5}(N)+O_{\varepsilon}(m^{2}(N^{\frac{23}{16}+\varepsilon}+N^{\frac{4}{3}+\frac{\delta}{6}+\varepsilon })(\log ( N\mathfrak{q}_{\pi}))^3).$$
 \end{enumerate}
\end{coro}
\begin{rem}
Following from Lemma \ref{4} in Appendix, we note that if $N$  satisfies some necessary congruence conditions,  then $\# \mathcal{J}_{1,3}(N)\asymp N^2/(\log N)^{3}$ and $\# \mathcal{J}_{2,5}(N)\asymp N^{3/2}/(\log N)^{5}$ which are the main terms in the above asymptotic formulas. 
\end{rem}
We have made a hypothesis regarding the automorphy of $\operatorname{Ad}\pi$ in the above corollary. In the following remark, we have presented several cases that meet the hypothesis.
\begin{rem}\label{pro-L}
\hspace*{\fill}
\begin{enumerate}[(1)]
\item 
If $m=2$, the adjoint lift of $\pi$ is automorphic. Furthermore, $\operatorname{Ad} \pi$ is cuspidal if and only if $\pi$ is not an automorphic induction of a Hecke character according to \cite{G-J-1978}. For $m>2$, the conjecture that the adjoint lift of $\pi$ is automorphic is widely open.
\item 

For a positive integer \(N_1\) and an even positive integer \(k_{1}\), we denote by \(H_{k_{1}}^*(N_1)\) the finite set of all primitive cusp forms of weight \(k_{1}\) with trivial nebentypus for the Hecke congruence group \(\Gamma_0(N_1)\). Each \(f \in H_{k_{1}}^*(N_1)\) has a Fourier expansion
\[
f(z)=\sum_{n=1}^{\infty} \lambda_f(n) n^{\frac{k_{1}-1}{2}} e(n z)
\]
in the upper half-plane. It is known that Newton and Thorne \cite{N-T-2019, N-T-2020} have proven that \( \mathrm{sym}^{j}f\) corresponds to \( \mathrm{sym}^j \pi_{f}\), which is a cuspidal automorphic representation of \(\mathrm{GL}_{j+1}\) over \(\mathbb{Q}\) for all \(j\ge 1\).
Using well-known identities: for all $j \ge 1$, 
\[
\mathrm{sym}^2\left(\mathrm{sym}^j \pi_{f}\right)=\bigoplus_{0 \le i \le j / 2} \mathrm{sym}^{2 j-4 i} \pi_{f} \quad  \textup{(see \cite[159]{F-H-1991})}
\]
and \begin{equation*}\label{sym-zero}
\left(\mathrm{sym}^j \pi_{f}\right)^{\otimes 2}:=\mathrm{sym}^j \pi_{f} \otimes \mathrm{sym}^j \pi_{f}=\bigoplus_{0 \le r \le j} \mathrm{sym}^{2 r} \pi_{f} \quad \textup{(see \cite[151]{F-H-1991})},
\end{equation*}
where $\mathrm{sym}^0 \pi_{f}$ is the 1-dimensional trivial representation, we have  $$\operatorname{Ad}\mathrm{sym}^j \pi_{f}=\bigoplus_{1 \le r \le j} \mathrm{sym}^{2 r} \pi_{f}.$$In this case, $\operatorname{Ad}\mathrm{sym}^j \pi_{f}$ is cuspidal and automorphic.
\item 
When \(m = 2\) and \(\mathrm{sym}^{j}\pi\) represents the \(j\)th symmetric power of the standard representation of \(\mathrm{GL}_2\), the automorphy of \(\mathrm{sym}^{j}\pi\) was established for \(j = 2\) by Gelbart and Jacquet \cite{G-J-1978}, and for \(j = 3, 4\) by Kim and Shahidi \cite{K-S-2002-2,K-S-2002-1}. Similar to (2), \(\operatorname{Ad}\mathrm{sym}^{2}\pi\) is also automorphic.
\end{enumerate}
\end{rem}

As an application, we present an analogue of the Sato--Tate conjecture. We restrict our analysis to forms with trivial nebentypus, ensuring that all Fourier coefficients are real. Additionally, we assume that \(f\in H_{k_{1}}^*(N_1)\) does not possess complex multiplication (CM), meaning there is no imaginary quadratic field \(K\) such that for \(p \nmid N_1\), \( p\) is inert in \(K\) if and only if \(\lambda_f(p)=0\).  Deligne's proof of the Weil conjectures \cite{D-1974} implies that for each prime \(p\), there exists an angle \(\theta_p \in[0, \pi]\) such that \(\lambda_f(p)=2 \cos \theta_p\).
The Sato--Tate conjecture asserts that if \(f\in H_{k_{1}}^*(N_1)\) is non-CM, then the sequence \(\left\{\theta_p\right\}\) is equidistributed in the interval \([0, \pi]\) with respect to the measure $$\mathrm{d} \mu_{\mathrm{ST}}:=(2 / \pi) \sin ^2 \theta \mathrm{d} \theta.$$ Equivalently,  we obtain that
\begin{equation}\label{pi}
\pi_{f, I}(x):=\#\left\{p \le x: \theta_p \in I, p \nmid N_1\right\} \sim \mu_{\mathrm{ST}}(I) \pi(x) \quad \text { as } x \rightarrow \infty,
\end{equation}
where \(\pi(x)=\#\{p \le x\}\) and \(I=[\alpha, \beta] \subseteq[0, \pi]\). Barnet-Lamb \emph{et al.} \cite{B-G-H-T-2011} proved the Sato--Tate conjecture. For all primitive characters $\chi$,   Thorner \cite{Thorner-2024} has proved that the twisted $L$-functions $L(s,\mathrm{sym}^{j}f\otimes\chi)$ have  no Landau--Siegel zeros. We provide a simpler proof in Lemma \ref{effec}.  Based on this property, we will show that the equidistribution of \(\theta_{p_1}\) remains undisturbed over the set \(\mathcal{J}_{1,3}(N)\). 
Moreover,  as an application of Theorem \ref{<3>}, we will also  investigate the distribution of \(\lambda_{f}(p_{1})\) associated with \((p_1,p_2,p_3)\) in the set \(\mathcal{J}_{1,3}(N)\) under Hypothesis S. We define 
\[
\mathcal{J}_{1,3,f, I}(N):=\{(p_1,p_2,p_3)\in \mathcal{J}_{1,3}(N) : \theta_{p_1} \in I, p_1 \nmid N_1\}.
\]
\begin{theo}\label{z11<2>}
Fix a non-CM form $f\in H_{k_{1}}^*(N_1)$.
\begin{enumerate}
 \item
If $N_1$ is squarefree, then
$$
\# \mathcal{J}_{1,3,f, I}(N)
=\mu_{\mathrm{ST}}(I)\left(\# \mathcal{J}_{1,3}(N)\right)+O\left(\frac{\log (k_1 N_1 \log N)}{ (\log N)^{1/2} }\left(\# \mathcal{J}_{1,3}(N)\right)\right).
$$
If $N_1$ is not squarefree, then
$$
\# \mathcal{J}_{1,3,f, I}(N)
=\mu_{\mathrm{ST}}(I)\left(\# \mathcal{J}_{1,3}(N)\right)+O\left(\frac{\log (k_1 N_1 \log N)}{ (\log N)^{1/4} }\left(\# \mathcal{J}_{1,3}(N)\right)\right).
$$
\item
 Under Hypothesis S,  we have
 $$
\# \mathcal{J}_{1,3,f, I}(N)
=\mu_{\mathrm{ST}}(I)\left(\# \mathcal{J}_{1,3}(N)\right)+O(N^{11/6+\delta/6}(\log N)^{3}).$$
\end{enumerate}

\end{theo}
It follows from the above result and \eqref{pi} that
\begin{equation}\label{jijijijiji}
\frac{\# \mathcal{J}_{1,3,f, I}(N)}{\# \mathcal{J}_{1,3}(N)} \sim \frac{\pi_{f, I}(N)}{\pi(N)} \quad \text { as } N \rightarrow \infty.
\end{equation}
From a different perspective, it follows from \eqref{jijijijiji} and Lemma \ref{4} that
\begin{equation}\label{eqeq1}
\frac{\# \mathcal{J}_{1,3,f, I}(N)}{ \pi_{f, I}(N)}\sim\frac{\# \mathcal{J}_{1,3}(N)}{\pi(N)}\sim \frac{\mathfrak{G}_{1,3}(N)N}{2(\log N)^{2}}\quad \text { as } N \rightarrow \infty. 
\end{equation}
The Goldbach conjecture states that every even number \(2N\) can be expressed as the sum of two odd primes.  Hardy and Littlewood \cite{H-L-1923} conjectured an asymptotic formula for the number of such representations as follow:
\begin{equation}\label{gold} 
\#\mathcal{J}_{1,2}(2N)\sim2 C_2 \frac{2N}{(\log 2N)^2} \prod_{\mathfrak{p}}\left(\frac{\mathfrak{p}-1}{\mathfrak{p}-2}\right),
\end{equation}
where $\mathfrak{p}$ is an odd prime divisor of $2N$, and
$$
C_2=\prod_{\varpi=3}^{\infty}\left(1-\frac{ 1}{(\varpi-1)^2}\right) .$$

Comparing the result \eqref{eqeq1} with the conjecture \eqref{gold}, it seems that the condition \(\theta_{p_1} \in I\)  does not affect the selection of \(p_2\) and \(p_3\) in \(\mathcal{J}_{1,3}(N)\). Building on results from the Waring--Goldbach problem, it is observed that each \(p_i\) ($1\le i \le k-1$) in \(\mathcal{J}_{k,u}(N)\) covers almost all primes. This suggests that the behavior of these sums, particularly in the context of Goldbach-type problems, can be analyzed probabilistically. Thus, we propose to approach the Goldbach conjecture from a probabilistic perspective, wherein the selection of primes in such sums can be viewed as random variables following certain statistical distributions.

Define the set
\begin{align*}\Omega:=\{&(1,1,N-2),(1,2,N-3),\dots,(1,N-2,1),(2,1,N-3),\\&(2,2,N-4),\dots,(2,N-3,1),\dots,(N-3,1,2),(N-3,2,1),(N-2,1,1)\}.\end{align*}
Let \(\mathcal{F}\) be the set of all subsets of \(\Omega\) and \(\phi\) be an empty set. According to \cite[Chapter 1]{Yong-Zhou-1999}, \((\Omega, \mathcal{F})\) forms a measurable space. Let \(\mathbf{P}\) be a map: \(\mathcal{F} \rightarrow[0,1]\), which is a probability measure on \((\Omega, \mathcal{F})\) satisfying
$$
\left\{\begin{array}{c}
\mathbf{P}(\phi)=0, \quad \mathbf{P}(\Omega)=1,\\
A_i \in \mathcal{F}~( i=1,2, \dots, 2^{\# \Omega}), \quad \mathbf{P}\left( A_i\right)=\frac{\# A_i}{\# \Omega}.
\end{array}\right.
$$
The triple \((\Omega, \mathcal{F}, \mathbf{P})\) forms a probability space. 
We denote two events
$$
A:=\{ (p_1,n_2,n_3)|p_1 \in \mathcal{P}, n_2,n_3 \in \mathbb{N}, N=p_1+n_2+n_3 \}
$$
and
$$
A^{\prime}:=\{ (n_1,p_2,p_3)|n_1 \in \mathbb{N}, p_2,p_3 \in \mathcal{P}, N=n_1+p_2+p_3 \}.
$$
It is evident that \(A \cap A^{\prime}\) is the set \(\mathcal{J}_{1,3}(N)\). Inspired by \eqref{eqeq1}, we propose an interesting conjecture related to the Goldbach conjecture.
\begin{hypothesis1}\label{hypor1}
As for $(\Omega, \mathcal{F}, \mathbf{P})$, the events $A$ and $A^{\prime}$ satisfy 
$$
\mathbf{P}(A \cap A^{\prime})\sim \mathbf{P}(A)\mathbf{P}(A^{\prime}).
$$
\end{hypothesis1}
By the Prime Number Theorem and Lemma \ref{4}, when $N$ is a sufficiently large odd integer, we derive that
$$\mathbf{P}(A)\sim\frac{1}{\log N},\quad\quad \mathbf{P}(A\cap A^{\prime})\sim\frac{\mathfrak{G}_{1,3}(N)}{(\log N)^{3}}.$$ Under Conjecture \ref{hypor1},    $$
\mathbf{P}(A^{\prime})\sim\frac{\mathbf{P}(A \cap A^{\prime})}{\mathbf{P}(A)}\sim\frac{\mathfrak{G}_{1,3}(N)}{(\log N)^{2}}.
$$
Since the number of samples in \(\Omega\) is given by \((N^{2}-3N+2)/2\), we can conclude that
\begin{equation}\label{rou}\# A^{\prime}\sim\frac{\mathfrak{G}_{1,3}(N)N^{2}}{2 (\log N)^{2}}.\end{equation}
 Also, we calculate the number of the samples in $A^{\prime}$, that is,
\begin{equation}\label{cal}
\#\mathcal{J}_{1,2}(N-4)+\#\mathcal{J}_{1,2}(N-5)+\dots+\#\mathcal{J}_{1,2}(4).
\end{equation}
It follows from \eqref{rou} and a rough estimate that 
\begin{equation}\label{rou1}
\sum_{n=2}^{(N-5)/2}\#\mathcal{J}_{1,2}(2n)\sim \frac{\mathfrak{G}_{1,3}(N)N^{2}}{2(\log N)^{2}}. 
\end{equation}
 If Conjecture  \ref{hypor1} holds,  we would obtain the estimate \eqref{rou1},  which shows that  the Goldbach conjecture is right on average. 
\begin{rem}
For \eqref{cal}, we can apply the asymptotic formula \eqref{gold} to derive a more precise estimate:
\[
\sum_{n=2}^{(N-5)/2} 2 C_2 \frac{2n}{(\log 2n)^2} \prod_{\mathfrak{p}}\left(\frac{\mathfrak{p}-1}{\mathfrak{p}-2}\right).
\]
Comparing this result with \eqref{rou1}, we see that Conjecture \ref{hypor1} provides additional support for the validity of the asymptotic formula \eqref{gold}.

\end{rem}
{\bf Structure of this paper.}  The structure of this paper is outlined as follows. In Section 2, we present two important propositions that are used to prove our main theorems. In Section 3, we provide some fundamental definitions and properties of \(L\)-functions along with several useful lemmas. Section 4 is dedicated to the proofs of Propositions \(\ref{712}\)--\(\ref{713}\), leveraging the Rankin–Selberg results, the Mellin transform, and the functional equations of \(L\)-functions. In Section 5, we establish Theorems \(\ref{<2>}\)--\(\ref{<3>}\) by employing the Hardy–Littlewood circle method in conjunction with Propositions \(\ref{712}\)--\(\ref{713}\). In Section 6, an analogue of the Sato–Tate conjecture is presented as an application of the methods used in Theorem \(\ref{<3>}\), denoted as Theorem \(\ref{z11<2>}\). Finally, Section \(\ref{appendix}\) provides several useful appendices.

In this paper, we investigate the apparent oscillatory properties of \(\lambda_{\pi}(n)\) in its interactions with the additive structure of integers. For example, in Theorem \(\ref{z11<2>}\), we observe that the equidistribution of the sequence \(\theta_p\) in the interval \([0, \pi]\) remains unaffected by the additive structure of integers associated with the Waring–Goldbach problem. We hope that the ideas presented in this paper (particularly the combination of concepts from both additive and multiplicative problems) may inspire further applications and advancements in related areas of research.

\section{Exponential estimates}
In relation to the Waring--Goldbach problem, we will utilize the Hardy--Littlewood circle method. This method involves transforming a discrete weighted sum into an integral over \([0,1)\) with respect to related exponential sums such as
 \begin{equation}\label{sum11}
\sum_{n \le x} \Lambda(n) a(n) e\left(n^k \alpha\right),  
\end{equation}
where $a(n)$ is an arithmetic function and $\alpha\in [0,1)$. Estimating such exponential sums is key to solve the problem. For example,  Fouvry and Ganguly \cite{F-G-2014}
obtained  
\begin{equation}\label{Fou-est}
\sum_{n \le x} \Lambda(n) \nu_f(n) e(n \alpha) \ll_{f} x \exp (-c \sqrt{\log x}),
\end{equation}
which was used in the proof of \eqref{V-1}.  
 Hou \emph{et al.} \cite{Hou-16} got
\begin{equation}\label{Hou-est}
\sum_{n \le x} \Lambda(n) A_F(n, 1) e(n \alpha) \ll_F x \exp (-c \sqrt{\log x})
\end{equation}
to  prove  \eqref{V-2}. The nonexistence of Landau--Siegel  zeros  plays an important role in estimating the sums \eqref{Fou-est} and \eqref{Hou-est}.    
 
In this paper, instead of relying on the nonexistence of Landau–Siegel zeros in the twisted \(L\)-functions, we consider the sum \eqref{sum11} for high ranks. Let \(\mathbb{T}\) denote the torus \(\mathbb{Z} \backslash \mathbb{R}\), and identify \(\mathbb{T}\) with the unit interval \([0, 1)\) throughout this paper. For any \(\alpha \in \mathbb{T}\), we provide estimates of \eqref{sum11} with \(a(n) = \lambda_{\pi}(n)\) that do not depend on the generalized Ramanujan conjecture (GRC, as described in Section \ref{3333}) or on the nonexistence of Landau–Siegel zeros in the twisted \(L\)-functions. This is the first time such estimates have been obtained.
\begin{prop}\label{712}
Let $\pi\in\mathfrak{E}_m^b$. For any real number  $\alpha \in \mathbb{T}$, suppose that
there exist two integers $a$ and $q$, such that
$$
\alpha=a/q +\lambda, \ \ |\lambda|\le \frac{1}{qQ}, \ \ (a,q)=1,\ \ 1\le q\le Q\le x^{k},
$$
where $Q$ is a large positive integer. 
For all $A>0$, if $q\le (\log x)^{A}$, then $$
  \sum_{p\le x}\lambda_{\pi}(p)e(p^{k}\alpha)\log p\ll_{\pi,A}
\bigg(1+\frac{x^{k}}{qQ}\bigg)q^{\frac{m+1}{2}}x \exp \left(-c \sqrt{\log x}\right),
  $$  where $c$ is an ineffective constant depending on $A$ and $\pi$.
\end{prop}
  
\begin{prop}\label{713}
 Let $\pi\in\mathfrak{E}_m$, $\varepsilon$ be a sufficiently small positive number.  For any real number  $\alpha \in \mathbb{T}$, suppose that
there exist two integers $a$ and $q$, such that
$$
\alpha=a/q +\lambda, \ \ |\lambda|\le \frac{1}{qQ}, \ \ (a,q)=1,\ \ 1\le q\le Q,
$$
where $Q$ is a large positive integer. 
Under Hypothesis S,   $$
  \sum_{p\le x}\lambda_{\pi}(p)e(p^{k}\alpha)\log p\ll
m^{2}\bigg(1+\frac{x^{k}}{qQ}\bigg)q^{\frac{1}{2}}x^{\delta+\varepsilon}(\log (x \mathfrak{q}_{\pi})) ^{3}+x^{\frac34}.
  $$
  \end{prop}
\begin{rem}
Compared to similar results previously, Proposition \ref{713} states the relation between the upper bound of the exponential  sum concerning $\lambda_{\pi}(p)$ and $\pi$  for the first time. 
\end{rem}

Our new approach can be applied not only to derive \eqref{V-1} and \eqref{V-2}, but also to extend similar results to higher ranks. Consider Dirichlet rational approximation, where there exist integers \(a\) and \(q\) such that
\begin{equation}\label{za}
\alpha = \frac{a}{q} + \lambda, \quad |\lambda| \le \frac{1}{qQ}, \quad (a, q) = 1, \quad 1 \le q \le Q,
\end{equation}
for any \(\alpha \in \mathbb{T}\), with \(Q\) being a large positive integer. When \(\alpha\) is in intervals around rational numbers \(a/q\) with large denominators (i.e., in the so-called minor arcs), we can apply results for the following exponential sums (refer to Lemmas \ref{Vino-Harman1}--\ref{Ren} below):
\[
\sum_{m \le x} \Lambda(m) e\left(m^k \alpha\right),
\]
where \(x\) is a large real number. In contrast, when \(\alpha\) lies in the major arcs, we use Propositions \ref{712}--\ref{713}. By employing two distinct estimations, one for the major arcs and the other for the minor arcs, we can bypass the limitations of previous methods in such problems, thereby obtaining Theorems \ref{<2>}--\ref{<3>}.

 \section{Background on $L$-functions}\label{3333}
\subsection{Standard $L$-functions}
\hspace*{\fill}\\

Let $m \ge 1$ be an integer.
 Fix $\pi=\otimes_{p} \pi_{p} \in \mathfrak{E}_{m}$, where $p$ ranges over all primes.
The standard function $L(s, \pi)$  is given by the Dirichlet series 
\begin{equation}\label{716}
\begin{aligned}
L(s, \pi) &=\prod_p L(s,\pi_{p})=\sum_{n=1}^{\infty} \frac{\lambda_\pi(n)}{n^s},\end{aligned}
\end{equation}
with both the series and the product converging absolutely for $\mathrm{Re}(s)>1$. For each $p$, the local factor $L(s,\pi_{p})$ is defined in terms of Satake parameters $\{\alpha_{1,\pi}(p),\dots,\alpha_{m,\pi}(p)\}$ by
$$
L(s,\pi_{p}):= \prod_{j=1}^m\left(1-\frac{\alpha_{j, \pi}(p)}{p^s}\right)^{-1}=\sum_{i=0}^{\infty} \frac{\lambda_\pi\left(p^i\right)}{p^{is}}.
$$
Let $\mathfrak{q}_{\pi}$ denote the conductor of $\pi$. For $1\le j\le m$,  we have $\alpha_{j,\pi}(p)\neq 0$ when $p\nmid \mathfrak{q}_{\pi}$. When $p\mid \mathfrak{q}_{\pi}$, the case might be that $\alpha_{j,\pi}(p)=0$.
At the archimedean place of $\mathbb{Q}$, there exist $m$ complex Langlands parameters $\mu_{\pi}(j)$, from which we define the Dirichlet series
\begin{equation}\label{lala}
L_{\infty}(s, \pi):=\pi^{-m s / 2} \prod_{j=1}^m \Gamma\left(\frac{s+\mu_\pi(j)}{2}\right).
\end{equation}
Let $\tilde{\pi}$ denote the contragredient
representation of $\pi \in \mathfrak{E}_m$, which is also an irreducible cuspidal automorphic representation in $\mathfrak{E}_m$. For each $p \le \infty$, we have the following equations
$$
\big\{\alpha_{j, \tilde{\pi}}(p):1\le j\le m\big\}=\big\{\overline{\alpha_{j, \pi}(p)}:1\le j\le m\big\}
$$
and
$$
\big\{\mu_{\tilde{\pi}}(j):1\le j\le m\big\}=\big\{\overline{\mu_\pi(j)}:1\le j\le m\big\}.
$$
For a nontrivial $\pi$, the complete $L$-function is defined by
$$
\Lambda(s, \pi):=\mathfrak{q}_{\pi}^{\frac{s}{2}} L(s, \pi) L_{\infty}(s, \pi).
$$
The complete $L$-function $\Lambda(s, \pi)$ has an analytic continuation  to the entire complex plane, and is of finite order. Moreover, it satisfies the functional equation
$$
\Lambda(s, \pi)=W_{\pi} \Lambda(1-s, \tilde{\pi}),
$$
where $W_{\pi}$ is the root number (a complex number of magnitude $1$).
For all primes $p$ and $1 \le j \le m$, it follows from \cite{K-2003} $(2 \le m \le 4)$ and \cite{L-S-1995} $(m \ge 5)$ that  the best known bound is
$$
\left|\alpha_{j, \pi}(p)\right| \le p^{\theta_{m}}  \quad \text { and } \quad -\mathrm{Re}(\mu_{\pi}(j))\le \theta_{m},
$$
where
\begin{equation}\label{saa}
\theta_{2}=\frac{7}{64}, \quad \theta_{3}=\frac{5}{14}, \quad \theta_{4}=\frac{9}{22}, \quad \theta_{m}=\frac{1}{2}-\frac{1}{m^{2}+1}\ (m \ge 5).
\end{equation}
The generalized Ramanujan conjecture predicts that
$$
\theta_{m}=0$$ for all $m\ge 1.$
The analytic conductor of $\pi$ is defined by
\begin{equation}\label{chi}
C(\pi, t):=\mathfrak{q}_{\pi} \prod_{j=1}^{m}\left(1+\left|i t+\mu_{\pi}(j)\right|\right), \quad C(\pi):=C(\pi, 0).
\end{equation}

Taking the logarithmic derivative for $L(s, \pi)$, we obtain that for $\mathrm{Re}(s)>1$,
$$
-\frac{L^{\prime}}{L}(s, \pi)=\sum_{p} \sum_{k=1}^{\infty} \frac{a_{\pi} (p^k) \log p}{p^{k s}}=\sum_{n=1}^{\infty} \frac{a_\pi(n) \Lambda(n)}{n^s},
$$
where
\begin{equation}\label{715}
a_\pi\left(p^k\right)=\sum_{j=1}^m \alpha_{j, \pi}(p)^k .
\end{equation}
\begin{lemma}\label{<1>}
Let $\pi \in \mathfrak{E}_{m}$, $m\ge 2$, and $\varepsilon$ be a  sufficiently small positive number. 
\begin{enumerate}
 \item 
Under Hypothesis S, 
\begin{equation}\label{add-1}
\sum_{n\le x} \Lambda(n)a_{\pi}(n)  \ll_{\varepsilon} mx^{\delta+\varepsilon} (\log x)^{2} \log \mathfrak{q}_{\pi}.
\end{equation}

 \item 
Under Hypothesis S and GRC, the upper bound is improved to 
\begin{equation}\label{add}
\sum_{n\le x} \Lambda(n)a_{\pi}(n)  \ll x^{\delta}(\log x)^{2}(\log \log (x^{m} C(\pi)))^{2}.
\end{equation}
\end{enumerate}
\end{lemma}
\begin{proof}
Let $y$ be a parameter satisfying the inequality $1 \le y \le x$, will be determined later. Denote $\phi(z)$ as the function with support on the interval $[0, x+y]$, satisfying the conditions $\phi(z)=1$ for $1 \le z \le x$, and $|\phi(z)| \le 1$ elsewhere.  By the Cauchy--Schwarz inequality,  \eqref{(1)} and the Prime Number Theorem in short intervals \cite{Huxley-72}, we can achieve a smoother expression by writing
\begin{align}\label{<1>-1}
\sum_{n\le x} \Lambda(n)a_{\pi}(n)=\sum_n \Lambda(n)a_{\pi}(n) \phi(n)+O(\sqrt{x y}), \quad x^{7/12}\ll y\ll x.
\end{align}
For $0 \le z \le x+y$, we take
$$
\phi(z)=\min \left(\frac{z}{y}, 1,1+\frac{x-z}{y}\right)
$$
and set  $\phi(z)=0$ for $z>x+y$. The Mellin transform of $\phi(z)$ satisfies
$$
\hat{\phi}(s)=\int_0^{x+y} \phi(z) z^{s-1} \mathrm{d} z \ll \frac{x^\sigma}{|s|} \min \left(1, \frac{x}{|s| y}\right)
$$
for $s=\sigma+i t, \frac{1}{2} \le \sigma \le 2$.
The sum on the right-hand side of \eqref{<1>-1} can be expressed as
$$
\sum_n \Lambda(n)a_{\pi}(n) \phi(n)=\frac{1}{2 \pi i} \int_{(2)}-\frac{L^{\prime}}{L}(s, \pi) \hat{\phi}(s) \mathrm{d} s,
$$
where the integral converges absolutely. Our operations will be carried out within the region defined by $$
 \mathcal{Z}:=\left\{s=\sigma+i t \mid \sigma= \delta+\varepsilon\right\}.
$$
 All the zeros of $L(s, \pi)$ are at least distance $\varepsilon$ from $\mathcal{Z}$. Therefore, it follows from \cite[(5.27)]{I-K-2004} and \cite[(5.28)]{I-K-2004}  that for $s \in \mathcal{Z}$,
$$
-\frac{L^{\prime}}{L}(s, \pi) \ll \varepsilon^{-1}(\log C(\pi)+m\log(|t|+3)).
$$
By moving the integration from the line $\mathrm{Re}(s)=2$ to $\mathcal{Z}$ and combining the above inequality, we apply Cauchy's theorem to derive the following:
\begin{align*}
\sum_n \Lambda(n)a_{\pi}(n) \phi(n)&=\frac{1}{2 \pi i} \int_{\mathcal{Z}}-\frac{L^{\prime}}{L}(s, \pi) \hat{\phi}(s) \mathrm{d}s\nonumber \\
&\ll \varepsilon^{-1}\int_{0}^{x/y}\frac{x^\sigma}{t+1} ( \log C(\pi)+m\log(t+3))\mathrm{d} t\nonumber\\
&\quad +\varepsilon^{-1}\int_{x/y}^{\infty}\frac{x^{\sigma+1}}{t^2y} ( \log C(\pi)+m\log(t+3))\mathrm{d} t \nonumber\\
&\ll \varepsilon^{-1} x^{\delta+\varepsilon} (m \log \mathfrak{q}_{\pi}+m\log(x/y+3))\log (x/y).
\end{align*}
We take $y=x\exp(-\sqrt{\log x})$ into the above inequality and combine it with \eqref{<1>-1}, thereby obtaining \eqref{add-1}. 

 Under GRC, the difference is that we can use \cite[(5.53)]{I-K-2004} to obtain that
$$
\sum_{n\le x} \Lambda (n)a_{\pi}(n) =\sum_{|\gamma| \le T} \frac{x^\rho-1}{\rho}+O\left(\frac{x}{T}(\log x) \log \left(x^m C(\pi)\right)\right),$$
where $\rho=\beta+i \gamma$ runs over the zeros of $L(s, \pi)$ in the critical strip of height up to $T$ with $1 \le T \le x$.  The implied constant is absolute.
Using  \cite[(5.27)]{I-K-2004} and choosing $T=x^{1-\delta}(\log x)\log(x^{m}C(\pi))$, we find that
$$
\sum_{n\le x} \Lambda(n)a_{\pi}(n)  \ll x^{\delta}(\log x)^{2}(\log \log (x^{m} C(\pi)))^{2}.
$$

\end{proof}

\subsection{Rankin--Selberg $L$-functions}
\hspace*{\fill}\\

Let $\pi^{\prime}=\otimes_{p} \pi_{p}^{\prime} \in \mathfrak{E}_{m^{\prime}}$ $(m^{\prime} \ge 1)$ and $\pi=\otimes_{p} \pi_{p} \in \mathfrak{E}_{m}$ $(m \ge 1)$. The Rankin--Selberg $L$-function $L\left(s, \pi \times \pi^{\prime}\right)$ associated to $\pi$ and $\pi^{\prime}$  is of the form
$$
L\left(s, \pi \times \pi^{ \prime}\right)=\prod_{p} L\left(s, \pi_{p} \times \pi_{p}^{\prime }\right)=\sum_{n=1}^{\infty} \frac{\lambda_{\pi \times \pi^{\prime}}(n)}{n^{s}}
$$
for $\mathrm{Re}(s)>1$. For each (finite) prime $p$, the inverse of the local factor $L\left(s, \pi_{p} \times \pi_{p}^{\prime}\right)$ is defined as a polynomial in $p^{-s}$ of degree not exceeding $ m m^{ \prime}$:
\begin{equation}\label{rsloc}
L\left(s, \pi_{p} \times \pi_{p}^{ \prime}\right)^{-1}:=\prod_{j=1}^{m} \prod_{j^{ \prime}=1}^{m^{ \prime}}\left(1-\frac{\alpha_{j, j^{\prime}, \pi \times \pi^{ \prime}}(p)}{p^{s}}\right),
\end{equation}
where $\alpha_{j, j^{ \prime}, \pi \times \pi^{ \prime}}(p)$ are  suitable complex numbers.  Given $\theta_{m}$ and $\theta_{m^{\prime}}$ as in (\ref{saa}), we have the pointwise bound
\begin{equation}\label{chirs}
\left|\alpha_{j, j^{ \prime}, \pi \times \pi^{ \prime}}(p)\right| \le p^{\theta_{m}+\theta_{m^{ \prime}}} \le p^{1-\frac{1}{m m^{ \prime}}}.
\end{equation}
If $p \nmid \mathfrak{q}_{\pi} \mathfrak{q}_{\pi^{ \prime}}$, we have 
\begin{equation}\label{df}
\bigg\{\alpha_{j, j^{ \prime}, \pi \times \pi^{ \prime}}(p): j \le m, j^{ \prime} \le m^{ \prime}\bigg\}=\bigg\{\alpha_{j, \pi}(p) \overline{\alpha_{j^{\prime }, \pi^{ \prime}}(p)}: j \le m, j^{ \prime} \le m^{\prime }\bigg\}.
\end{equation}
At the archimedean place of $\mathbb{Q}$, there are $m m^{\prime}$ complex Langlands parameters $\mu_{\pi\times \pi^{\prime}}\left(j, j^{\prime}\right)$ from which we define the Dirichlet series
$$
L_{\infty}\left(s, \pi \times \pi^{\prime}\right):=\pi^{-\frac{m m^{\prime} s}{2}} \prod_{j=1}^{m} \prod_{j^{\prime}=1}^{m^{\prime}} \Gamma\left(\frac{s+\mu_{\pi \times \pi^{\prime}}(j, j^{\prime})}{2}\right) .
$$
These parameters satisfy the pointwise bound
$$
\mathrm{Re}\left(\mu_{\pi \times \pi^{\prime}}\left(j, j^{\prime}\right)\right) \ge-\theta_{m}-\theta_{m^{\prime}} .
$$
Let $\mathfrak{q}_{\pi \times \pi^{\prime}}$ denote the conductor of $\pi \times \pi^{\prime}$. As with $L(s, \pi\times\pi^{\prime})$, we define the analytic conductor $C\left(\pi \times \pi^{\prime}\right)$ by
\begin{equation}\label{4242}
C\left(\pi \times \pi^{\prime}, t\right):=\mathfrak{q}_{\pi \times \pi^{\prime}} \prod_{j=1}^{m} \prod_{j^{\prime}=1}^{m^{\prime}}\left(1+\left|i t+\mu_{\pi \times \pi^{\prime}}\left(j, j^{\prime}\right)\right|\right), \ C\left(\pi \times \pi^{\prime}\right):=C\left(\pi \times \pi^{\prime}, 0\right).
\end{equation}
The combined work of  \cite{B-H-1997} alongside \cite[Appendix]{S-T-2019}  shows that
\begin{equation}\label{4241}
C\left(\pi \times \pi^{\prime}, t\right) \ll C\left(\pi \times \pi^{\prime}\right)(1+|t|)^{m^{\prime} m}, \quad C\left(\pi \times \pi^{\prime}\right) \le e^{O\left(m^{\prime} m\right)} C(\pi)^{m^{\prime}} C\left(\pi^{\prime}\right)^{m} .
\end{equation}
Let $r_{\pi \times \pi^{\prime}} \in\{0,1\}$ be the order of the pole of $L\left(s, \pi \times \pi^{\prime}\right)$ at $s=1$. We have  $r_{\pi \times \pi^{\prime}}=1$ if and only if $\pi^{\prime}=\tilde{\pi}$. The function $\Lambda\left(s, \pi \times \pi^{\prime}\right)=(s(s-1))^{r_{\pi \times \pi^{\prime}}} \mathfrak{q}_{\pi \times \pi^{\prime}}^{s / 2} L\left(s, \pi \right.$  $\left.\times \pi^{\prime}\right) L_{\infty}\left(s, \pi \times\pi^{\prime}\right)$ is entire of order one. There exists a complex number $W\left(\pi \times \pi^{\prime}\right)$ of modulus one such that $$\Lambda\left(s, \pi \times \pi^{\prime}\right)=W\left(\pi \times \pi^{\prime}\right) \Lambda\left(1-s, \tilde{\pi} \times \tilde{\pi}^{\prime}\right).$$ 
Furthermore, $L(s, \pi \times \tilde{\pi})$ extends to the complex plane with a simple pole at $s=1$, hence
\begin{equation}\label{rs}
\sum_{n\le x} \lambda_{\pi \times \widetilde{\pi}}(n) \sim x \mathrm{Res}_{s=1} L(s, \pi \times \widetilde{\pi}) \ll x,
\end{equation}
as a consequence of a standard Tauberian argument. Taking the logarithmic derivative for $L(s, \pi\times\pi^{\prime})$, we obtain that, for $\mathrm{Re}(s)>1$,
$$
-\frac{L^{\prime}}{L}(s, \pi\times\pi^{\prime})=\sum_{p} \sum_{k=1}^{\infty} \frac{a_{\pi\times\pi^{\prime}} (p^k) \log p}{p^{k s}}=\sum_{n=1}^{\infty} \frac{a_{\pi\times\pi^{\prime}}(n) \Lambda(n)}{n^s},$$
where
$$
a_{\pi\times\pi^{\prime}}\left(p^k\right)=\sum_{j=1}^m \sum_{j^{\prime}=1}^{m^{\prime}} \alpha_{j, j^{\prime}, \pi \times \pi^{\prime}}(p)^k.
$$
Building upon Shahidi's non-vanishing result for $L(s, \pi \times \tilde{\pi})$ at $\mathrm{Re}(s)=1$ (see \cite{S-1981}), then
\begin{equation}\label{(1)}
\sum_{n \le x}\Lambda(n)\left|a_{\pi}(n)\right|^{2}\ll\sum_{n \le x} \Lambda(n) a_{\pi \times \widetilde{\pi}}(n) \sim x.
\end{equation}

\subsection{Twists}
\hspace*{\fill}\\

Let $\chi$ be a primitive Dirichlet character with conductor $q$. Fix $\pi=\otimes_{p} \pi_{p} \in \mathfrak{E}_{m}$ with $m\ge 2$. By \cite{C-2008}, the twisted $L$-function $L(s,\pi\otimes\chi)$ equals
$$
\sum_{n=1}^{\infty} \frac{\lambda_{\pi \otimes \chi}(n) }{n^s}
=\prod_p \prod_{j=1}^m \left(1-\frac{\alpha_{j, \pi \otimes \chi }(p)}{p^s}\right)^{-1}.
$$
Furthermore, by \eqref{df}, if $p \nmid q$, then
$$
\begin{aligned}
\left\{\alpha_{j, \pi\otimes\chi }(p): 1 \le j \le m\right\}
=\left\{ \chi(p)\alpha_{j, \pi }(p): 1 \le j \le m\right\} .
\end{aligned}
$$
Hence, for $\mathrm{Re}(s)>1$, we  have
$$
\begin{aligned}
\sum_{n=1}^{\infty} \frac{\chi(n)\lambda_{\pi }(n) }{n^s} &=\prod_p  \prod_{j=1}^m \left(1-\frac{\chi(p)\alpha_{j, \pi }(p)}{p^s}\right)^{-1}=L(s, \pi\otimes\chi ) \prod_{p \mid q} \prod_{j=1}^m\left(1-\frac{\alpha_{j, \pi \otimes \chi}(p)}{p^s}\right).
\end{aligned}
$$
 
 We  use the following zero-free region  along with standard contour integration techniques to prove Proposition \ref{712}.
 \begin{lemma}\label{th6.1}Let $\pi\in\mathfrak{E}_{m}^{b}$ and $Q\ge 3$. There exists a constant $c_\pi>0$, depending effectively on $\pi$, such that for all primitive Dirichlet characters $\chi (\bmod\thinspace q)$ with $q \le Q$, with at most one exception, the $L$-function $L(s, \pi \otimes \chi)$ is non-zero in the region
$$
\mathrm{Re}(s) \ge 1-\frac{c_\pi}{\log (Q(3+|\mathrm{Im}(s)|))} .
$$
If the exceptional character $\chi_1$ exists, then
$\chi_1$ is quadratic.
$L\left(s, \pi \otimes\chi_1\right)$ has exactly one zero $\beta_1$ in this region, and $\beta_1$ is both real and simple.
For all $\varepsilon_{0}>0$, there exists an ineffective constant $c_\pi(\varepsilon_{0})>0$ depending on $\pi$ and $\varepsilon_0$ such that $\beta_1 \le 1-$ $c_\pi(\varepsilon_{0}) Q^{-\varepsilon_{0}}$.
\end{lemma}
\begin{proof}
See \cite[Theorem 4.1]{J-L-W-2021}.
\end{proof}

\section{Proof of Propositions \ref{712}--\ref{713}}
\subsection{Reduction}
\hspace*{\fill}\\

Define the sum
$$
\sum_{n \le x}\Lambda(n)g(n) e(n^{k} \alpha),
$$
where $ g(n)$ is an arithmetic function. According to the Dirichlet approximation theorem, for any $\alpha \in[0,1)$, there exist integers $a$ and $q$ such that  $\alpha$ can be expressed as \eqref{za}, where $Q$ is a parameter to be chosen later. Combining (\ref{za}) with partial summation, the above sum is bounded by
\begin{equation}\label{za7}
\left(1+\frac{x^{k}}{q Q}\right) \max _{1 \le t \le x}\left|\sum_{n \le t}\Lambda(n) g(n) e\left(\frac{an^{k}}{q}\right)\right| .
\end{equation}
Due to the orthogonality of characters, we find that
\begin{equation}\label{5.4}
\begin{aligned}
\sum_{n \le t}\Lambda(n) g(n) e\left(\frac{a n^{k}}{q}\right)
&=\sum_{d \mid q} \sum_{\substack{h=1 \\
(h, q / d)=1}}^{q} e\left(\frac{a d^{k} h^{k}}{q}\right) \sum_{\substack{l \le t / d  \\
l \equiv h(\bmod q / d)}}\Lambda(d l) g(d l) \\
&=\sum_{d \mid q} \sum_{h=1 }^{q/d}  e\left(\frac{a d^{k} h^{k}}{q}\right) \sum_{l \le t / d}\Lambda(d l) g(d l)\frac{1}{\varphi(q / d)} \sum_{\substack{\chi(\bmod q / d)\\
(h, q / d)=1}} \bar{\chi}(h) \chi(l)\\&
\ll \sum_{d \mid q} \frac{1}{\varphi(q / d)} \max _{\chi\left(\bmod q / d\right)}\bigg|\sum_{l \le t / d} \Lambda(d l)g(d l) \chi(l)\bigg| \\
& \quad \times \sum_{\chi\left(\bmod q / d\right)}\bigg|\sum_{\substack{h=1 \\
(h, q / d)=1}}^{q / d} \bar{\chi}(h) e\left(\frac{ad^k h^k }{q}\right)\bigg|,
\end{aligned}
\end{equation}
where $\varphi$ is the Euler function. By the Cauchy--Schwarz inequality and Parseval's identity, the innermost sum over $\chi$ is bounded by
\begin{align}\label{5.3}
& \varphi\left(q / d\right)^{1/2}\left(\sum_{\chi\left(\bmod q / d\right)}\bigg|\sum_{\substack{h=1 \\
(h, q / d)=1}}^{q / d} \bar{\chi}(h) e\left(\frac{ad^k h^k }{q}\right)\bigg|^2\right)^{1/2} \nonumber\\
= & \varphi\left(q / d\right)^{1/2}\left(\sum_{\substack{h_1=1 \\
\left(h_1, q / d\right)=1}}^{q / d} \sum_{\substack{h_2=1 \\
\left(h_2, q / d\right)=1}}^{q / d} e\left(\frac{ad^k\left(h_1^k-h_2^k\right) }{q}\right) \sum_{\substack{\chi\left(\bmod q / d\right)}} \bar{\chi}\left(h_1\right) \chi\left(h_2\right)\right)^{1/2} \nonumber\\
= & \varphi\left(q / d\right)\left(\sum_{\substack{h_1=1 \\
\left(h_1, q / d\right)=1}}^{q / d} \sum_{\substack{h_2=1 \\
h_2 \equiv h_1\left(\bmod q / d\right)}}^{q / d}e\left(\frac{ad^k\left(h_1^k-h_2^k\right)}{q}\right)\right) ^{1/2}=\varphi\left(q / d\right)^{3/2}.
\end{align}
Inserting the estimates \eqref{5.4} and \eqref{5.3} into \eqref{za7}, then
\begin{equation}\label{zong}
\sum_{n \le x}\Lambda(n) g(n) e\left(n^{k} \alpha\right) \ll \left(1+\frac{x^{k}}{q Q}\right) \max _{1 \le t \le x}\left\{ q^{1 / 2} \sum_{d \mid q} \frac{1}{d^{1 / 2}} \max _{\chi (\bmod q / d)}\bigg|\sum_{l \le t / d}\Lambda(d l) g(d l) \chi(l)\bigg|\right\}.
\end{equation}

\subsection{Proof of Proposition \ref{712}}
\hspace*{\fill}\\

It follows from \eqref{716} and \eqref{715} that 
\begin{equation}\label{label0}
\sum_{p\le x}\lambda_{\pi}(p)e(p^{k}\alpha)\log p=\sum_{n \le x}\Lambda(n)a_{\pi}(n) e(n^{k} \alpha)-\sum_{p^{k_1}\le x,
 k_1\ge 2 } a_{\pi}(p^{k_1})e(p^{k_1}\alpha)\log p.
\end{equation}
Using the Cauchy--Schwarz inequality, (\ref{rs}) and (\ref{(1)}), the left side of the above equation is equal to
\begin{equation}\label{label1}
\sum_{n \le x}\Lambda(n)a_{\pi}(n) e(n^{k} \alpha)+O(x^{3/4}).
\end{equation}
Let $g(n)=a_{\pi}(n)$ in \eqref{zong}. Then the above summation is 
\begin{equation}\label{label1-1}
\ll \left(1+\frac{x^{k}}{q Q}\right) \max _{1 \le t \le x}\left\{ q^{1 / 2} \sum_{d \mid q} \frac{1}{d^{1 / 2}} \max _{\chi (\bmod q / d)}\bigg|\sum_{l \le t / d}\Lambda(d l) a_{\pi}(d l) \chi(l)\bigg|\right\}.
\end{equation}
Let $l=l_{1} l_{2} $ with $l_{1} \mid d^{\infty}$ and $\left(l_{2}, l_{1}d\right)=1.$ By the multiplicative property of $a_{\pi}(dl)$, the formula in the brace is
\begin{align}\label{zong-1}
 \ll &q^{1 / 2} \sum_{d \mid q,\ d\neq 1} \frac{1}{d^{1 / 2}} \max _{\chi(\bmod q / d)}\bigg|\sum_{l_{1} \mid d^{\infty},\ l_{1} \le t/d}\Lambda(dl_1) a_{\pi}(dl_{1})\chi(l_{1})\bigg|\nonumber\\
&+q^{1 / 2}\max _{\chi(\bmod q)}\bigg|\sum_{l_{2} \le t} \Lambda(l_{2}) a_{\pi}\left(l_{2}\right) \chi\left(l_{2}\right)\bigg|.
\end{align}
Due to the property of $\Lambda(n)$ and the upper bound of $|a_{\pi}(n)|$, the first item of \eqref{zong-1} is
\begin{align}\label{zong-2}
 \ll q^{1 / 2}t^{1/2} \sum_{p^{k} \mid q,\ k\ge 1} \frac{1}{p^{k / 2}}\bigg|\sum_{p^{k_1}\le t/p^{k},\ k_1\ge 1}\log p\bigg| \ll (qt)^{1/2}(\log q) \log t. 
\end{align}
If $\chi^{\prime}(\bmod\thinspace q^{\prime})$ is the primitive Dirichlet character that induces $\chi(\bmod\thinspace q)$, then $q^{\prime}|q$. For the second item of \eqref{zong-1}, by \eqref{saa}, we get
\begin{equation}\label{718}
\left|\sum_{l_{2} \le t }\Lambda(l_{2}) a_{\pi}\left(l_{2}\right) \chi\left(l_{2}\right)-\sum_{l_{2} \le t } \Lambda(l_{2}) a_{\pi \otimes \chi^{\prime}}(l_{2})\right| \le \sum_{p \mid q} \sum_{p^k \le t, \ k \ge 1}\left|a_\pi\left(p^k\right)\right| \log p \ll_{\varepsilon}t^{\theta_m+\varepsilon} .
\end{equation}
Without loss of generality, assume that $\chi^{\prime}$ is the exceptional character $\chi_{1}$ in Lemma \ref{th6.1}, and let $\beta^{\prime}$ be the corresponding exceptional zero. It follows from \eqref{4241} and \eqref{4242} that the modulus of the conductor of $\pi \otimes \chi^{\prime}$ is smaller than $\mathfrak{q}_{\pi}q^{\prime m}$. We apply \cite[(5.52)]{I-K-2004} to the twisted 
$L$-function $L(s,\pi\otimes\chi^{\prime})$ and conclude that
$$
\sum_{l_{2} \le t } \Lambda(l_{2}) a_{\pi \otimes \chi^{\prime}}(l_{2})=-\frac{t^{\beta^{\prime}}}{\beta^{\prime}}+O\bigg(\mathfrak{q}_{\pi}^{\frac{1}{2}}q^{\prime\frac{m}{2}}t \exp \left(-\frac{c_\pi}{2} \sqrt{\log t}\right)\bigg).
$$
 Since Lemma \ref{th6.1} gives the bound $\beta^{\prime} \le 1-c_\pi(\varepsilon_0) q^{-\varepsilon_0}$, where  $\varepsilon_0>0$ is a fixed constant chosen later,  for all $q^{\prime}\le q\le (\log t)^{A}$, we have
 \begin{align*}
\sum_{l_{2} \le t } \Lambda(l_{2}) a_{\pi \otimes \chi^{\prime}}(l_{2})
\ll_{\pi}  t\exp \left(-c_\pi(\varepsilon_0)(\log t) q^{-\varepsilon_0}\right)+
tq^{\frac{m}{2}} \exp \left(-\frac{c_\pi}{4} \sqrt{\log t}\right).
 \end{align*}
 Choose $\varepsilon_0=1/(2A)$. Then inserting \eqref{label1-1}--\eqref{718} and the above inequality into \eqref{label1}, we get 
  $$ \sum_{p\le x}\lambda_{\pi}(p)e(p^{k}\alpha)\log p \ll_{\pi,A}
\bigg(1+\frac{x^{k}}{qQ}\bigg)q^{\frac{m+1}{2}}x \exp \left(-c \sqrt{\log x}\right),
  $$
  where $c=\min \left\{\frac{c_\pi}{8},\frac{c_{\pi}(1/(2A))}{2}\right\}$.
\subsection{Proof of Proposition \ref{713}}
\hspace*{\fill}\\

Under Hypothesis  S, Lemma \ref{<1>} gives  that
 $$ \sum_{l_{2} \le t } \Lambda(l_{2}) a_{\pi \otimes \chi^{\prime}}(l_{2})
\ll  mt^{\delta+\varepsilon} (\log t)^{2} \log (\mathfrak{q}_{\pi}q^{m}).$$ 
Inserting \eqref{label1-1}--\eqref{718} and the above inequality into \eqref{label1}, we get
   $$
  \sum_{p\le x}\lambda_{\pi}(p)e(p^{k}\alpha)\log p\ll
m^{2}\bigg(1+\frac{x^{k}}{qQ}\bigg)q^{\frac{1}{2}}x^{\delta+\varepsilon}(\log x\mathfrak{q}_{\pi})^{3}+x^{3/4}.
  $$

\section{Proof of Theorems \ref{<2>}--\ref{<3>}}
\subsection{Preliminaries and lemmas}
\hspace*{\fill}\\

The initial step of the Hardy--Littlewood circle method is to divide the unit interval into the major arcs and the minor arcs. The major arcs $\mathfrak{M}$ and minor arcs $\mathfrak{m}$ are defined as follows:
\begin{equation}\label{dearcs}
\ \ \ \ \ \ \ \ \mathfrak{M}=\bigcup_{q\le P}\bigcup_{1\le a\le q \atop (a, q)=1}\mathfrak{M}(q, a), \ \ \ \mathfrak{M}(q, a)=\Big[\frac{a}{q}-\frac{1}{qQ},\frac{a}{q}+\frac{1}{qQ}\Big], \ \ \ \ \ \mathfrak{m}=\big[1/Q,1+1/Q\big]\setminus \mathfrak{M},
\end{equation}
where $1<2P<Q$, $PQ=N$, $P$, $Q$ are  parameters to be chosen later.
Let
\begin{align}\label{TH-def}
T_{k,1}(N, \alpha)=\sum_{p_{1} \le N^{1/k}} \lambda_{\pi}(p_{1}) e(p_{1}^{k} \alpha), \quad
H_{k, u^{\prime}}(N, \alpha)=\sum_{p_{u^{\prime}} \le N^{1/k}}  e(p_{u^{\prime}}^{k} \alpha) \quad (2 \le u^{\prime}\le u).
\end{align}
Using  Parseval's identity, we write
\begin{align}\label{maj-min}
\sum_{(p_1,p_2,\dots,p_u)\in\mathcal{J}_{k,u}(N)}\lambda_{\pi}(p_{1}) &=\int_0^1 T_{k,1}(N, \alpha) \prod_{u^{\prime}=2}^{u} H_{k,u^{\prime}}(N, \alpha)e(-N \alpha) \mathrm{d} \alpha\nonumber\\
&=\int_{\mathfrak{M}} T_{k,1}(N, \alpha) \prod_{u^{\prime}=2}^{u} H_{k,u^{\prime}}(N, \alpha)e(-N \alpha) \mathrm{d} \alpha \nonumber\\
&\quad+\int_{\mathfrak{m}} T_{k,1}(N, \alpha) \prod_{u^{\prime}=2}^{u} H_{k,u^{\prime}}(N, \alpha)e(-N \alpha) \mathrm{d} \alpha .
\end{align}

To control the integral on the minor arcs, we need certain estimates concerning exponential sums over primes.
\begin{lemma}\label{Vino-Harman1} 
Suppose that $\alpha$ is a real number and there exist integers $a$ and $q$ such that $\alpha$ 
satisfies
\begin{equation}\label{Fa-dis}
\alpha=\frac{a}{q}+\lambda, \quad 1 \le a \le q, \quad(a, q)=1, \quad|\lambda| \le \frac{1}{q^2}.
\end{equation}
\begin{enumerate}
  \item For $k=1$,  
  $$\sum_{m\le x}\Lambda(m)e(m^{k}\alpha)\ll x (\log x)^3\left(\frac{1}{q}+\frac{1}{x^{2 / 5}}+\frac{q}{x}\right)^{1 / 2}.
$$
  \item For all  $k \ge 2$ and $\varepsilon_1>0$, 
$$
\sum_{m\le x}\Lambda(m)e(m^{k}\alpha) \ll x^{1+\varepsilon_1}\left(\frac{1}{q}+\frac{1}{x^{1 / 2}}+\frac{q}{x^k}\right)^{1 / K^2},
$$
where $K=2^{k-1}$.
\end{enumerate}
\end{lemma}
\begin{proof}
For $k=1$, one can see \cite{Vaughan-1977}. For $k\ge 2$, one can see \cite[Theorem 1]{H-G-1981}.
\end{proof}
Sometimes, for $k\ge 2$, it is better to have the factor $x^{\varepsilon_1}$ above replaced by some power of $\log x$. This can be done by a simple modification.
\begin{lemma}\label {Harman2}
Let  $k \ge 2$.  Suppose that $\alpha$ is a real number and there exist integers $a$ and $q$ such that $\alpha$ 
satisfies
\eqref{Fa-dis}. Then
$$
\sum_{m\le x}\Lambda(m)e(m^{k}\alpha) \ll x (\log x)^{c^{\prime}}\left(\frac{1}{q}+\frac{1}{x^{1 / 2}}+\frac{q}{x^k}\right)^{1 / (2 K^2)},
$$
where $c^{\prime}$ is a positive constant that depends only on $k$, and $K$ is defined in Lemma \ref{Vino-Harman1}.
\end{lemma}

\begin{lemma}\label {Ren}
 Let $k\ge1$, $\gamma_k=1/2+\log k/\log 2$. Suppose that $\alpha$ is a real number and there exist integers $a$ and $q$ such that $\alpha$ 
satisfies
\eqref{Fa-dis}. Then
$$
\sum_{m\le x}\Lambda(m)e(m^{k}\alpha)\ll (d(q))^{\gamma_k}\left\{x^{1 / 2} \sqrt{q\left(1+|\lambda| x^k\right)}+x^{4 / 5}+\frac{x}{\sqrt{q\left(1+|\lambda| x^k\right)}}\right\} (\log x)^{c^{\prime\prime}},
$$
where $d(n)$ is the divisor function  and $c^{\prime\prime}$ is an absolute positive constant.
\end{lemma}
\begin{proof}
See \cite[Theorem 1.1]{R-2005}.
\end{proof}

\subsection{Proof of Theorem \ref{<2>}}
\subsubsection{Major arcs}

As for $r_{\pi}(N)$, we use the orthogonality to derive that
\begin{equation}\label{h1u}
\begin{aligned}
\int_0^1\left|H_{1, u^{\prime}}(N, \alpha)\right|^2 \mathrm{d}\alpha \le \sum_{p_{u^{\prime}_{1}}\le N}\sum_{p_{u^{\prime}_{2}}\le N} \int_0^1 e((p_{u^{\prime}_{1}}-p_{u^{\prime}_{2}})\alpha)\mathrm{d}\alpha \le \sum_{p\le N}1\ll N/\log N.
\end{aligned}
\end{equation}
Let  $P\le (\log N)^{100B}$, where $P$ be chosen later. By applying the Cauchy--Schwarz inequality, the above inequality, partial summation and Proposition \ref{712},
\begin{align}\label{dfdf}
\int_{\mathfrak{M}} & T_{1,1}(N, \alpha)  H_{1, 2}(N, \alpha)H_{1, 3}(N, \alpha) e(-N \alpha)\mathrm{d} \alpha \nonumber\\
\quad & \ll\left(\max _{\alpha \in\mathfrak{M}}\left|T_{1,1}(N, \alpha)\right|\right)\left\{\int_0^1\left|H_{1,2}(N, \alpha)\right|^2 \mathrm{d} \alpha\right\}^{1 / 2} \left\{\int_0^1\left|H_{1,3}(N, \alpha)\right|^2 \mathrm{d} \alpha\right\}^{1 / 2}\nonumber \\
& \ll_{\pi,B}P^{\frac{m+1}{2}}N^{2}
    \exp \left(- c\sqrt{\log N}\right),
\end{align}
where $c$ is the constant in Proposition \ref{712}.

Regarding $w_{\pi}(N)$, we apply \cite[Theorem 3]{Blomer-06} to give
\begin{align}\label{Four-value}
\int_0^1\left|H_{2,2}(N, \alpha)H_{2,3}(N, \alpha)\right|^2 \mathrm{d} \alpha \le \sum_{p_2^{2}+p_3^{2}=p_2^{\prime2}+p_3^{\prime2}\atop p_2, p_3 ,p_{2}^{\prime },p_3^{\prime}\le N^{1/2} }1\ll N\log N.
\end{align}
Let $P\le (\log N)^{100B}$, where $P$ be chosen later. It follows from the Cauchy--Schwarz inequality, the above inequality, partial summation and Proposition \ref{712} that
\begin{align}\label{dfdfdf}
\int_{\mathfrak{M}} T_{2,1}(N, \alpha)&H_{2,2}(N, \alpha)H_{2,3}(N, \alpha)H_{2,4}(N, \alpha)H_{2,5}(N, \alpha)e(-N \alpha)\mathrm{d} \alpha\nonumber\\
 \quad \ll&\left(\max _{\alpha \in\mathfrak{M}}\left|T_{2,1}(N, \alpha)\right|\right)\left\{\int_0^1\left|H_{2,2}(N, \alpha)H_{2,3}(N, \alpha)\right|^2 \mathrm{d} \alpha\right\}^{1 / 2}\nonumber\\
  &\quad \times\left\{\int_0^1\left|H_{2,4}(N, \alpha)H_{2,5}(N, \alpha)\right|^2 \mathrm{d} \alpha\right\}^{1 / 2}\nonumber\\
\quad \ll&_{\pi,B}P^{\frac{m+1}{2}}N^{\frac{3}{2}} \exp \left(- \frac{c}{2}\sqrt{\log N}\right),
\end{align}
where $c$ is the constant in Proposition \ref{712}.
\subsubsection{Minor arcs}
By the Polya--Vinogradov inequality and \cite[(5.48)]{I-K-2004}, 

\begin{align}\label{maj12}
\int_0^1\left|T_{k,1}(N, \alpha)\right|^2 \mathrm{d} \alpha &\le \sum_{p_{1}\le N^{1/k}}\sum_{p_{2}\le N^{1/k}}\lambda_{\pi}(p_{1}) \bar{\lambda}_{\pi}(p_{2}) \int_0^1 e((p_{1}^{k}-p_{2}^{k})\alpha) \mathrm{d} \alpha\nonumber\\
&\le \sum_{p\le N^{1/k}}|\lambda_{\pi}(p)|^{2 }\ll N^{1/k}m^{2}(\log(N\mathfrak{q}_{\pi}))^{2}.
\end{align}
Following from the Cauchy--Schwarz inequality, \eqref{h1u} and the above inequality, we have
\begin{align}\label{min-1}
\int_{\mathfrak{m}} & T_{1,1}(N, \alpha)  H_{1, 2}(N, \alpha)H_{1, 3}(N, \alpha) e(-N \alpha)\mathrm{d} \alpha\nonumber \\
\quad & \le\left(\max _{\alpha \in\mathfrak{m}}\left|H_{1,2}(N, \alpha)\right|\right)\left\{\int_0^1\left|T_{1,1}(N, \alpha)\right|^2 \mathrm{d} \alpha\right\}^{1 / 2}\left\{\int_0^1\left|H_{1,3}(N, \alpha)\right|^2 \mathrm{d} \alpha\right\}^{1 / 2} \nonumber \\
& \ll \left(\max _{\alpha \in\mathfrak{m}}\left|H_{1,2}(N, \alpha)\right|\right)Nm\log(N\mathfrak{q}_{\pi}).
\end{align}
Using the first inequality in Lemma \ref{Vino-Harman1} and partial summation, we get
\begin{equation*}
\max _{\alpha \in\mathfrak{m}}\left|H_{1,2}(N, \alpha)\right|\ll (Q^{1/2}N^{1/2}+P^{-1/2}N+N^{4/5})(\log N)^{3}.
\end{equation*}
Inserting this bound into \eqref{min-1}, then
\begin{equation}\label{dfdf1}
\begin{aligned}
\int_{\mathfrak{m}} & T_{1,1}(N, \alpha) H_{1,2}(N, \alpha) H_{1,3}(N, \alpha) e(-N \alpha) \mathrm{d} \alpha \ll N^{2}m\log(N\mathfrak{q}_{\pi})(P^{-1/2}+N^{-1/5})(\log N)^{3}.
\end{aligned}
\end{equation}

For $w_{\pi}(N)$, we obtain that
\begin{align}\label{douve}
\int_0^1\left|T_{2,1}(N, \alpha)H_{2,3}(N, \alpha)\right|^2 \mathrm{d} \alpha& \ll \sum_{p_1^{2}+p_3^{2}=p_1^{\prime2}+p_3^{\prime2}\atop p_1, p_3, p_{1}^{\prime }, p_3^{\prime }\le N^{1/2}}|\lambda_{\pi}(p_1)\lambda_{\pi}(p_1^{\prime})|\nonumber\\
&\ll\sum_{p_1\le N^{1/2}}|\lambda_{\pi}(p_{1})|^{2}
\sum_{p_1^{\prime}\le N^{1/2}}
\sum_{p_1^{2}-p_1^{\prime2}=(p_3^{\prime}+p_3)(p_3^{\prime}-p_3)\atop p_3, p_3^{\prime }\le N^{1/2}}1\nonumber\\
&\ll Nm^{2}\log^{2}(N\mathfrak{q}_{\pi})\log N. 
\end{align}
Treat $w_{\pi}(N)$ in the same way as the estimate for  $r_{\pi}(N)$, to obtain
\begin{align}\label{25mino}
&\int_{\mathfrak{m}} T_{2,1}(N, \alpha)H_{2,2}(N, \alpha)H_{2,3}(N, \alpha)H_{2,4}(N, \alpha)H_{2,5}(N, \alpha)e(-N \alpha)\mathrm{d} \alpha\nonumber\\
 & \le\left(\max _{\alpha \in\mathfrak{m}}\left|H_{2,2}(N, \alpha)\right|\right)\left\{\int_0^1\left|T_{2,1}(N, \alpha)H_{2,3}(N, \alpha)\right|^2 \mathrm{d} \alpha\right\}^{1 / 2}\nonumber\\&\quad\times\left\{\int_0^1\left|H_{2,4}(N, \alpha)H_{2,5}(N, \alpha)\right|^2 \mathrm{d} \alpha\right\}^{1 / 2} \nonumber\\
& \ll \left(\max _{\alpha \in\mathfrak{m}}\left|H_{2,2}(N, \alpha)\right|\right)Nm\log(N\mathfrak{q}_{\pi})\log N.
\end{align}
In accordance with Lemma \ref {Harman2} and by partial summation,  we give the following upper bound 
$$
\max _{\alpha \in\mathfrak{m}}\left|H_{2,2}(N, \alpha)\right|\ll N^{1/2}(\log N)^{c^{\prime}}\left(\frac{1}{P}+\frac{1}{N^{1/4}}+\frac{Q}{N}\right)^{1/8}.
$$
Inserting the above estimate into \eqref{25mino}, 
\begin{align}\label{dfdf2}
&\int_{\mathfrak{m}} T_{2,1}(N, \alpha)H_{2,2}(N, \alpha)H_{2,3}(N, \alpha)H_{2,4}(N, \alpha)H_{2,5}(N, \alpha)e(-N \alpha)\mathrm{d} \alpha \nonumber\\
&\ll N^{3/2}m\log(N\mathfrak{q}_{\pi})(\log N)^{c^{\prime}+1}\bigg(P^{-1 / 8} +N^{-1/32}\bigg).
\end{align}
\subsubsection{Choosing parameters}
For $r_{\pi}(N)$,  we choose  $P=(\log N)^{2B+8}$. Then by \eqref{dfdf}, \eqref{dfdf1} and \eqref{maj-min}, we  have
$$
r_\pi(N)=\sum_{(p_1,p_2,p_3)\in\mathcal{J}_{1,3}(N)} \lambda_{\pi}\left(p_1\right) \ll_{\pi, B} N^{2}(\log N)^{-B}.
$$

For $w_{\pi}(N)$,   we choose  $P=(\log N)^{8B+8c^{\prime}+16}$. It follows from \eqref{dfdfdf}, \eqref{dfdf2} and \eqref{maj-min} that
$$
w_\pi(N)=\sum_{(p_1,p_2,p_3,p_4,p_5)\in\mathcal{J}_{2,5}(N)} \lambda_{\pi}\left(p_1\right) \ll_{\pi, B} N^{3/2}(\log N)^{-B}.$$
 \subsection{Proof of Theorem \ref{<3>}}\label{pr-th-1.4}   
 \hspace*{\fill}\\
 
We use Proposition \ref{713}  to improve the upper bounds in \eqref{dfdf} and \eqref{dfdfdf} respectively.  Then 
\begin{equation}\label{dfdf3}
\begin{aligned}
\int_{\mathfrak{M}} & T_{1,1}(N, \alpha)  H_{1, 2}(N, \alpha)H_{1, 3}(N, \alpha) e(-N \alpha)\mathrm{d} \alpha \ll  Pm^2 N^{1+\delta+\varepsilon} (\log ( N\mathfrak{q}_{\pi}))^3+N^{7/4}/\log N
\end{aligned}
\end{equation}
and 
\begin{align}\label{dfdf4}
\int_{\mathfrak{M}}& T_{2,1}(N, \alpha)H_{2,2}(N, \alpha)H_{2,3}(N, \alpha)H_{2,4}(N, \alpha)H_{2,5}(N, \alpha)e(-N \alpha)\mathrm{d} \alpha\nonumber\\
&\ll Pm^2 N^{1+\delta/2+\varepsilon} (\log ( N\mathfrak{q}_{\pi}))^3\log  N+N^{11/8}\log N.
\end{align}

For $r_\pi(N)$ under Hypothesis S, we choose $P=N^{\frac{2}{3}(1-\delta)}$ in \eqref{dfdf1} and \eqref{dfdf3}  to get
$$
r_\pi(N) \ll m^2 N^{5/3+\delta/3+\varepsilon}(\log ( N\mathfrak{q}_{\pi}))^3.
$$

For $w_\pi(N)$ under Hypothesis S, put $P=N^{\frac{1}{3}(1-\delta)}, \ Q=N/P$ and $P^{\prime}=N^{1 / 3}, \ Q^{\prime}=N/P^{\prime}$. On the minor arcs, one notes that each $\alpha \in \mathfrak{m}$ can be written as \eqref{dearcs} with $P<q \le Q$. By Dirichlet's rational approximation, each $\alpha \in \mathfrak{m}$ can also be written as \eqref{za} with $1 \le q \le Q^{\prime}$ and $|\lambda| \le 1 / q Q^{\prime}$. Let $\mathfrak{n}$ be the subset of $\alpha \in \mathfrak{m}$ such that $P^{\prime}<q \le Q^{\prime}$ and $|\lambda| \le 1 / qQ^{\prime}$. We now apply the second inequality of Lemma \ref{Vino-Harman1} and partial summation to get
\begin{equation}\label{min-1.1}
\max _{\alpha \in\mathfrak{n}}\left|H_{2,2}(N, \alpha)\right| \ll N^{1/2+\varepsilon}\left(\frac{1}{P^{\prime}}+\frac{1}{N^{1 / 4}}+\frac{Q^{\prime}}{N}\right)^{1 / 4} \ll N^{7/16+\varepsilon} .
\end{equation}
For $\alpha \in \mathfrak{m} \backslash \mathfrak{n}$, we have
$$
P<q \le P^{\prime}, \quad|\lambda| \le \frac{1}{q Q^{\prime}}
$$
or
$$
q \le P, \quad \frac{1}{q Q}<|\lambda| \le \frac{1}{q Q^{\prime}}.
$$
In either case, there holds
$$
N^{\frac{1}{6}(1-\delta)}\ll \sqrt{\min \left(P, \frac{N}{Q}\right)} \ll \sqrt{q\left(1+|\lambda| N\right)} \ll \sqrt{P^{\prime}+\frac{N}{Q^{\prime}}} \ll N^{1 /6}.
$$
Then  Lemma \ref{Ren} gives
\begin{align}\label{min-1.2}
\max _{\alpha \in \mathfrak{m} \backslash \mathfrak{n}}\left|H_{2,2}(N, \alpha)\right| & \ll\left\{N^{1 / 4} \sqrt{q\left(1+|\lambda| N\right)}+N^{2 / 5}+\frac{N^{1/2}}{\sqrt{q\left(1+|\lambda|N\right)}}\right\} d(q)^{2} (\log N)^{c^{\prime\prime}} \nonumber\\
& \ll N^{5 / 12+\varepsilon} +N^{\frac{1}{3}+\frac{\delta}{6}+\varepsilon} .
\end{align}
Combing \eqref{min-1.1} and \eqref{min-1.2}, we derive that
\begin{align}\label{min-1.3}
 \max _{\alpha \in \mathfrak{m}}\left|H_{2,2}(N, \alpha)\right| \ll N^{\frac{7}{16}+\varepsilon} +N^{\frac{1}{3}+\frac{\delta}{6}+\varepsilon}.
\end{align}
It follows from  \eqref{min-1.3}, \eqref{25mino}, \eqref{dfdf4} and \eqref{maj-min} that
$$
w_{\pi}(N)\ll m^{2}(N^{\frac{23}{16}+\varepsilon}+N^{\frac{4}{3}+\frac{\delta}{6}+\varepsilon})(\log ( N\mathfrak{q}_{\pi}))^3.
$$ 
This finishes the proof of Theorem \ref{<3>}.
\begin{rem}\label{com}
 For $w_\pi(N)$ under Hypothesis S, if we refrain from subdividing the minor arcs further and instead individually apply Lemma \ref{Vino-Harman1} or Lemma \ref{Ren}, we obtain the following two results.
 \begin{enumerate}
   \item  Using the second inequality of Lemma \ref{Vino-Harman1} and choosing $P=N^{\frac25(1-\delta)}$ , we have 
\begin{equation*}
\max _{\alpha \in\mathfrak{m}}\left|H_{2,2}(N, \alpha)\right| \ll N^{1/2+\varepsilon}\left(\frac{1}{P}+\frac{1}{N^{1 / 4}}\right)^{1 / 4} \ll N^{2/5+\delta/10+\varepsilon}.
\end{equation*}
Inserting this into \eqref{25mino} and combing \eqref{dfdf4}, we obtain
$$
w_{\pi}(N)\ll m^{2}N^{\frac{7}{5}+\frac{\delta}{10}+\varepsilon}(\log ( N\mathfrak{q}_{\pi}))^3.
$$ 
   \item Similarly, using  Lemma \ref{Ren} and choosing $P=N^{1/2-\delta/3}$, we have 
   \begin{equation*}
\max _{\alpha \in\mathfrak{m}}\left|H_{2,2}(N, \alpha)\right| \ll N^{1/2+\varepsilon}\left(N^{1/4}P^{-1/2}+N^{-1 / 10}\right) \ll N^{1/2+\delta/6+\varepsilon} .
\end{equation*}
Then 
$$
w_{\pi}(N)\ll m^{2}N^{3/2+\delta/6+\varepsilon}(\log ( N\mathfrak{q}_{\pi}))^3.
$$ 
 \end{enumerate}By directly comparing the aforementioned two results, it becomes evident that
 $N^{\frac{7}{5}+\frac{\delta}{10}+\varepsilon}$ and $N^{3/2+\delta/6+\varepsilon}$ are larger than $(N^{\frac{23}{16}+\varepsilon}+N^{\frac{4}{3}+\frac{\delta}{6}+\varepsilon})$. Hence, the rationale behind further subdividing the minor arcs is stronger.
 \end{rem}

\section{Preliminaries and proof of Theorem \ref{z11<2>}}
\subsection{The modular $L$-functions}
\hspace*{\fill}\\

Fix a non-CM form $f\in H_{k_{1}}^*(N_1)$. Let $$
f(z)=\sum_{n=1}^{\infty}  \lambda_f(n) n^{(k_{1}-1) / 2} e(n z)
$$
be its normalized Fourier expansion at the cusp $\infty$. For each prime $p$,  let $\theta_{p} \in[0, \pi]$ be the unique angle such that $\lambda_{f}(p)=2 \cos \theta_{p}$. The modular $L$-function $L(s, f)$ associated with $f$ has the Euler product representation
$$
L(s, f)=\sum_{n=1}^{\infty} \frac{\lambda_{f}(n)}{n^{s}}=\prod_{p}\left(1-\frac{\lambda_{f}(p)}{p^{s}}+\frac{\chi_{0}(p)}{p^{2 s}}\right)^{-1}, \quad \mathrm{Re}(s)>1,
$$
where $\chi_{0}$ is the trivial Dirichlet character modulo $N_1$. Let's rewrite the Euler product as
$$
L(s, f)=\prod_{p \mid q}\left(1-\frac{\left(-\lambda_{p} p^{-\frac{1}{2}}\right)}{p^{s}}\right) \prod_{p \nmid q} \prod_{j=0}^{1}\left(1-\frac{e^{i(2 j-1) \theta_{p}}}{p^{s}}\right)^{-1}, \quad \mathrm{Re}(s) >1,
$$
where $\lambda_{p} \in\{-1,1\}$ is the eigenvalue of the Atkin-Lehner operator $|_{k_1} W\left(Q_{p}\right)$.
According to Deligne \cite{D-1974}, for any prime  $p$, there are two  numbers $\alpha_{f}(p)$ and $\beta_{f}(p)$ such that
\begin{equation}\label{223}
\alpha_{f}(p)=e^{i \theta_{p}},\quad \beta_{f}(p)=e^{-i \theta_{p}}
\end{equation}
and
$$
\lambda_{f}(p)=\alpha_{f}(p)+\beta_{f}(p)
$$
with $\theta_{p} \in[0, \pi]$.
Deligne proved the Ramanujan-Petersson conjecture, which states that
$$
\left|\lambda_{f}(n)\right| \le d(n).
$$

For each $j \ge 1$ and $\mathrm{Re}(s)>1$, the $j$th symmetric power $L$-function attached to $f$ is  defined by
\begin{equation}\label{714}
L\left(s, \mathrm{sym}^{j} f\right)=\sum_{n=1}^{\infty} \frac{\lambda_{\mathrm{sym}^{j} f}(n)}{n^{s}}=\prod_{p \mid q} \prod_{m=0}^{j}\left(1-\frac{\alpha_{ \mathrm{sym}^{j}f,m }(p)}{p^{s}}\right)^{-1} \prod_{p \nmid q} \prod_{m=0}^{j}\left(1-\frac{e^{i(2 m-j) \theta_{p}}}{p^{s}}\right)^{-1}.
\end{equation}
The values $\alpha_{ \mathrm{sym}^{j}f,m }(p)$ can be determined by using \cite[Appendix]{S-T-2019} and  \cite[Appendix]{D-G-P-T-2020}. Note that $L\left(s,\mathrm{sym}^{1} f\right)=L(s, f)$ and $L\left(s,\mathrm{sym}^{0} f\right)=\zeta(s)$. 
It is well known that $\lambda_{\mathrm{sym}^{j} f}(n)$ is a real multiplicative function. Via \eqref{714}, one easily verify  that
\begin{equation}\label{lamda-uj}
\lambda_{\mathrm{sym}^{j} f}(p)=\frac{\sin(n+1)\theta_p}{\sin \theta_p}=U_{j}\left(\cos \theta_{p}\right), \quad p \nmid N_1.
\end{equation}
Let $\Gamma_{\mathbb{R}}(s):=\pi^{-s/2}\Gamma(s/2)$ and $\Gamma_{\mathbb{C}}(s)=\Gamma_{\mathbb{R}}(s) \Gamma_{\mathbb{R}}(s+1)$. With these definitions, we obtain that
\begin{equation}\label{<66>}
L\left(s,\left(\mathrm{sym}^j f\right)_{\infty}\right)= \begin{cases}q_{\mathrm{sym}^j f} \prod_{m=1}^{(j+1) / 2} \Gamma_{\mathbb{C}}\left(s+\left(m-\frac{1}{2}\right)(k-1)\right), & \text { if } j \text { is odd}, \\ q_{\mathrm{sym}^j f} \Gamma_{\mathbb{R}}(s+r) \prod_{m=1}^{j / 2} \Gamma_{\mathbb{C}}(s+m(k-1)), & \text { if } j \text { is even},\end{cases}
\end{equation}
where  $q_{\mathrm{sym}^{j} f}$ is  a suitable integer,  and  $r=0$ if $j \equiv 0\ (\bmod\ 4)$, and $r=1$ if $j \equiv 2\ (\bmod\ 4)$. Let $C(\mathrm{sym}^{j}f)$ be   the analytic conductor of $L(s, \mathrm{sym}^{j}f)$.  For any  primitive Dirichlet character $\chi$, we give a  proof  that  the twist $L$-function $L(s, \mathrm{sym}^{j}f\otimes \chi)$ has no Landau--Siegel zeros. 
\begin{lemma}\label{effec}
Let  $Q$ be a large positive number and $\chi $ be a primitive Dirichlet character modulo $q$ with $q\le Q$. Define
\[
y_{f} = 
\begin{cases}
0 & \text{if } f \text{ has squarefree level or } f \text{ corresponds}\\
   &\text{with a non-CM elliptic curve over } \mathbb{Q}, \\
2 & \text{otherwise}.
\end{cases}
\]
There exists a constant $c_{9}>0$ such that if $n \ge 1$, then $L\left(s, \operatorname{sym}^n f \otimes \chi \right) \neq 0$ for
$$
\mathrm{Re}(s) \ge  1-\frac{c_{9}}{n^{2+y_{f}} \log (nk_1N_1Q (3+|\operatorname{Im}(s)|))}.
$$
\end{lemma}
\begin{proof}
The proof can be found in \cite[Corollary 1.3]{Thorner-2024}.
\end{proof}

\subsection{The  Sato--Tate conjecture}
\hspace*{\fill}\\

Serre's modular version of the Sato--Tate conjecture posits that if $f$ is non-CM, then the sequence $\left\{\theta_{p}\right\}$ is uniformly distributed in the interval $[0, \pi]$ with respect to the measure $\mathrm{d} \mu_{S T}:=(2 / \pi) \sin ^{2} \theta \mathrm{d} \theta$. More explicitly, let $I_{0}$ be a subinterval of $[-2,2]$, and for a positive real number $x$, define the set
$$
N_{I_{0}}(f, x):=\#\left\{p \le x: \mathrm{gcd}(p, N_1)=1, \lambda_f(p) \in I_{0}\right\}.
$$
The Sato--Tate conjecture states that for a fixed non-CM form $f \in H_{k_{1}}^*(N_1)$, we have
$$
\lim _{x \rightarrow \infty} \frac{N_{I_{0}}(f, x)}{\pi(x)}=\int_{I_{0}} \mu_{\infty}(t) \mathrm{d} t,
$$
where 
$$
\mu_{\infty}(t):= \begin{cases}\frac{1}{\pi} \sqrt{1-\frac{t^2}{4}}, & \text { if } t \in[-2,2], \\ 0, & \text { otherwise. }\end{cases}
$$
If $\chi_I$ is the indicator function of the interval $I=[\alpha, \beta]\subseteq[0,\pi]$, then 
$$
\sum_{\substack{p  \le x \\
\lambda_{f}(p) \in [2\cos\beta, 2\cos\alpha]}}1=\sum_{p \le x} \chi_I\left(\theta_p\right) .
$$
 Thorner \cite{Thorner-2014} approximate $\chi_I$ with a differentiable function using the following construction formula in \cite{V-2004}.
\begin{lemma}\label{71le1}
Let $R$ be a positive integer, and let $a, b, \delta^{\prime} \in \mathbb{R}$ satisfy
$$
0<\delta^{\prime}<1 / 2, \quad \delta^{\prime} \le b-a \le 1- \delta^{\prime} .
$$
Then there exists an even periodic function $g(y)$ with period 1 satisfying
\begin{enumerate}
 \item
 $g(y)=1$ when $y \in\left[a+\frac{ \delta^{\prime}}{2} , b-\frac{ \delta^{\prime}}{2}\right]$,
 \item
 $g(y)=0$ when $y \in\left[b+\frac{ \delta^{\prime}}{2}, 1+a-\frac{ \delta^{\prime}}{2}\right]$,
 \item
 $0 \le g(y)\le 1$ when $y$ is in the rest of the interval $\left[a-\frac{ \delta^{\prime}}{2} , 1+a-\frac{ \delta^{\prime}}{2} \right]$, and
 $g(y)$ has the Fourier expansion
$$
g(y)=b-a+\sum_{n=1}^{\infty}\left(a_n \cos (2 \pi n y)+b_n \sin (2 \pi n y)\right),
$$
where for all $n \ge 1$,
$$
\left|a_n\right|,\left|b_n\right| \le \min \left\{2(b-a), \frac{2}{n \pi}, \frac{2}{n \pi}\left(\frac{R}{\pi n  \delta^{\prime}}\right)^R\right\}.
$$
\end{enumerate}
\end{lemma}

\subsection{Proof of Theorem \ref{z11<2>}}
\hspace*{\fill}\\

If $\chi_I$ is the indicator function of the interval $I=[\alpha, \beta]\subseteq[0,\pi]$, then 
$$
\# \mathcal{J}_{1,3,f, I}(N)=\sum_{\substack{(p_1,p_2,p_3)\in\mathcal{J}_{1,3}(N) \\
\lambda_{f}(p_1) \in [2\cos\beta, 2\cos\alpha]}} 1=\sum_{(p_1,p_2,p_3)\in\mathcal{J}_{1,3}(N)} \chi_I\left(\theta_{p_1}\right) .
$$

Let $g(\theta_{p_1})$ be defined as in Lemma \ref{71le1} with $a=\frac{\alpha}{2 \pi}-\frac{\delta^{\prime}}{2}$ and $b=\frac{\beta}{2 \pi}+\frac{\delta^{\prime}}{2}$. Thorner \cite{Thorner-2014} constructed
$g^{+}(\theta_{p_1} ; I, \delta^{\prime})=g\Big(\frac{\theta_{p_1}}{2 \pi}\Big)+g\Big(-\frac{\theta_{p_1}}{2 \pi}\Big)$ as a pointwise upper bound for $\chi_I(\theta_{p_1})$. By repeating this construction with $a=\frac{\alpha}{2 \pi}+\frac{\delta^{\prime}}{2}$, and $b=\frac{\beta}{2 \pi}-\frac{\delta^{\prime}}{2}$, we can get a lower bound $g^{-}(\theta_{p_1} ; I, \delta^{\prime})$ for $\chi_I(\theta_{p_1})$. Thus, according to  Lemma \ref{71le1}, we can express $g^{\pm}(\theta_{p_1} ; I, \delta^{\prime})$ in terms of the basis of Chebyshev polynomials of the second kind $\{U_n(\cos (\theta_{p_1})) \}_{n=0}^{\infty}$ as follows:
\begin{equation}\label{z1<2>}
g^{\pm}(\theta_{p_1} ; I, \delta^{\prime})=a_0^{\pm}(I, \delta^{\prime})-a_2^{\pm}(I, \delta^{\prime})+\sum_{n=1}^{\infty}\left(a_n^{\pm}(I, \delta^{\prime})-a_{n+2}^{\pm}(I, \delta^{\prime})\right) U_n(\cos (\theta_{p_1})),
\end{equation}
where $a_n^{\pm}(I, \delta^{\prime})$ represents the $n$-th Fourier coefficient in the cosine expansion of $g^{\pm}(\theta_{p_1} ; I, \delta^{\prime})$.
Furthermore, we derive 
\begin{equation}\label{cos1}
\begin{aligned}
\left|a_0^{\pm}(I,  \delta^{\prime})-a_2^{\pm}(I,  \delta^{\prime})-\mu_{S T}(I)\right| & \ll \delta^{\prime} ,
\end{aligned}
\end{equation}
and for $n\ge 1$,
\begin{equation}\label{cos2}
\left|a_n^{\pm}(I,  \delta^{\prime})\right|,\left|a_{n+2}^{\pm}(I,  \delta^{\prime})\right| \le \min \left\{2(b-a), \frac{2}{n \pi}, \frac{2}{n \pi}\left(\frac{R}{\pi n \delta^{\prime}}\right)^R\right\}.\end{equation}
Following the method in \cite{T-2014}, when summing $g^{\pm}\left(\theta_{p_1} ; I, \delta^{\prime}\right)$ over $\mathcal{J}_{1, 3}(N)$,  we can interchange the order of summation
by choosing $R$ to ensure absolute convergence. Thus by \eqref{z1<2>}, we obtain the following expression:
$$\begin{aligned}
\sum_{(p_1,p_2,p_3)\in\mathcal{J}_{1,3}(N)} g^{\pm}(\theta_{p_1} ; I, \delta^{\prime})
&=\sum_{(p_1,p_2,p_3)\in\mathcal{J}_{1,3}(N)}(a_0^{\pm}(I, \delta^{\prime})-a_2^{\pm}(I, \delta^{\prime}))\nonumber\\
&\quad+\sum_{(p_1,p_2,p_3)\in\mathcal{J}_{1,3}(N)}\bigg(\sum_{n=1}^{\infty}\left(a_n^{\pm}(I, \delta^{\prime})-a_{n+2}^{\pm}(I, \delta^{\prime})\right) U_n(\cos (\theta_{p_1}))\bigg)\nonumber\\
&=(a_0^{\pm}(I, \delta^{\prime})-a_2^{\pm}(I, \delta^{\prime}))\sum_{(p_1,p_2,p_3)\in\mathcal{J}_{1,3}(N)}1\nonumber\\
&\quad+\sum_{n=1}^{\infty}\left(a_n^{\pm}(I, \delta^{\prime})-a_{n+2}^{\pm}(I, \delta^{\prime})\right)\sum_{(p_1,p_2,p_3)\in\mathcal{J}_{1,3}(N)}U_n(\cos (\theta_{p_1})).\end{aligned}
$$
It follows from \eqref{cos1}, \eqref{cos2}, \eqref{lamda-uj} and the above equation that
\begin{equation}\label{jijiji}
\begin{aligned}
\sum_{(p_1,p_2,p_3) \in\mathcal{J}_{1,3}(N)} \chi_I\left(\theta_{p_1}\right)
=(\mu_{\mathrm{ST}}(I)+O(\delta^{\prime}))\bigg(\sum_{(p_1,p_2,p_3)\in\mathcal{J}_{1,3}(N)} 1\bigg)+O(A),
\end{aligned}
\end{equation}
where
\begin{align}\label{A-defi}
A=\sum_{n=\delta^{\prime-1}}^{\infty}\frac{4}{n \pi}\left(\frac{R}{\pi n \delta^{\prime}}\right)^R \sum_{(p_1,p_2,p_3)\in\mathcal{J}_{1,3}(N)}\lambda_{\mathrm{sym}^{n}f}(p_1)+\sum_{n=1}^{\delta^{^\prime-1}} \frac{2}{n\pi}\sum_{(p_1,p_2,p_3)\in\mathcal{J}_{1,3}(N)}\lambda_{\mathrm{sym}^{n}f}(p_1).
\end{align}

Let $\pi_{f} \in \mathfrak{E}_{2}$ correspond to $f$. It is known from \cite{N-T-2019} and \cite{N-T-2020} that $\mathrm{sym}^{j} f$ corresponds to $\mathrm{sym}^{j}\pi_{f} \in \mathfrak{E}_{j+1}$. Utilizing the method outlined in the proof of Theorem \ref{<3>}, \eqref{add} with the zero-free region in Lemma \ref{effec}, and \cite[(6.5)]{T-2021}, if the level of $f$ is squarefree, we derive the inequality
$$
\sum_{(p_1,p_2,p_3)\in\mathcal{J}_{1,3}(N)} \lambda_{\mathrm{sym}^{n}f}(p_1)\ll (n+1) N^{2} \exp\bigg(-\frac{c_{9}\log N}{n^2 \log (nk_1N_1\log N )} \bigg).
$$
Based on Lemma \ref{71le1},  we choose $\delta^{\prime-1}=(\log N)^{1/2} / \log (k_1 N_1 \log N)$, and $R=2$ in \eqref{A-defi}, then
$$
A\ll \frac{N^{2} (\log N)^{1/2}}{\log (k_1 N_1 \log N)}\exp\bigg(-\frac{c_{9}\log^{2} (k_1 N_1 \log N)}{ \log ( k_1N_1 \log^{2} N )} \bigg).
$$
Inserting this upper bound of $A$ into \eqref{jijiji}, we have
$$
\sum_{(p_1,p_2,p_3)\in\mathcal{J}_{1,3}(N)} \chi_I\left(\theta_{p_1}\right)
=\mu_{\mathrm{ST}}(I)\left(\# \mathcal{J}_{1,3}(N)\right)+O\left(\frac{\log (k_1 N_1 \log N)}{ (\log N)^{1/2} }\left(\# \mathcal{J}_{1,3}(N)\right)\right).
$$
Similar to the discussion above, if the level of $f$ is not squarefree, we choose $\delta^{\prime-1}=(\log N)^{1/4} / \log (k_1 N_1 \log N)$ and $R=2$ in \eqref{A-defi}  to obtain 
$$
\sum_{(p_1,p_2,p_3)\in\mathcal{J}_{1,3}(N)} \chi_I\left(\theta_{p_1}\right)
=\mu_{\mathrm{ST}}(I)\left(\# \mathcal{J}_{1,3}(N)\right)+O\left(\frac{\log (k_1 N_1 \log N)}{ (\log N)^{1/4} }\left(\# \mathcal{J}_{1,3}(N)\right)\right).
$$
Under Hypothesis S, and using the method outlined in the proof of Theorem \ref{<3>}, \eqref{add}, and \cite[(6.5)]{T-2021}, we derive the inequality:
$$
\sum_{(p_1,p_2,p_3)\in\mathcal{J}_{1,3}(N)} \lambda_{\mathrm{sym}^{n}f}(p_1)\ll (n+1) N^{5/3+\delta/3}(\log N)^3\log (k_1N_1n).
$$
In \eqref{A-defi}, we choose $\delta^{\prime}=RN^{\delta/6-1/6}\log( k_1N_1N)$ and $R=2$ based on Lemma \ref{71le1},  then
$$
A \ll N^{11/6+\delta/6}(\log N)^{3}.
$$
Substituting this upper bound of $A$ into \eqref{jijiji}, we have
$$
\sum_{(p_1,p_2,p_3)\in\mathcal{J}_{1,3}(N)} \chi_I\left(\theta_{p_1}\right)
=\mu_{\mathrm{ST}}(I)\left(\# \mathcal{J}_{1,3}(N)\right)+O(N^{11/6+\delta/6}(\log N)^{3}).
$$

  \section{Appendix}\label{appendix}
In this section, for a better understanding of the results in Theorem \ref{twist}, we provide  asymptotic formulas for the ternary Goldbach conjecture, as well as  an asymptotic formula for the number of the sums of five square primes. 
\begin{lemma}\label{4}
 For a large integer $N$, we have
 \begin{enumerate}
  \item 
$$
\# \mathcal{J}_{1,3}(N)
=\mathfrak{G}_{1,3}(N)\frac{N^2}{2(\log N)^{3}}+O\left(\frac{N^{2}\log \log N}{(\log N)^{4}}\right),
$$
where $\mathfrak{G}_{1,3}(N)$ is the singular series defined in \eqref{singular series} and satisfies $\mathfrak{G}_{1,3}(N)\ge 1/2$ for odd $N$.
 \item
\begin{equation*}\label{5}
\# \mathcal{J}_{2,5}(N)=
\mathfrak{G}_{2,5}(N)\frac{\pi^2N^{3/2}}{24(\log N)^{5}}+O\left(\frac{N^{3/2}\log \log N}{(\log N)^{6}}\right),
\end{equation*}
where $\mathfrak{G}_{2,5}(N)$  is absolutely convergent and   satisfies $\mathfrak{G}_{2,5}(N)\gg1$ for $N\equiv5\ ({\rm mod}\ 24)$.
\end{enumerate}
\end{lemma}
\begin{proof}
One can see the detail proof in \cite[Theorem 11]{Hua-65}. Here we take the first equation as an example to illustrate the main ideas. By the  Hardy--Littlewood circle method, we consider\begin{align*}
\# \mathcal{J}_{1,3}(N) &=\left\{\int_{\mathfrak{M}}+\int_{\mathfrak{m}}\right\}H_{1, 1}(N, \alpha) H_{1, 2}(N, \alpha)H_{1, 3}(N, \alpha)e(-N \alpha) \mathrm{d} \alpha,
\end{align*}
where $\mathfrak{M}$ and $\mathfrak{m}$ are defined in \eqref{dearcs}. For the integral on the minor arcs, we follow a similar approach as in \eqref{min-1}, but replace $T_{1, 1}(N, \alpha)$ with $H_{1, 1}(N, \alpha)$. Regarding the major arcs, due to the sparsity of prime numbers, we seek an equation that establishes a connection between prime numbers and integers to derive an asymptotic formula. Consider
\begin{align*}
H_{1, i}(N, \alpha)= &  \sum_{\substack{h=1 \\
(h, q)=1}}^q e\left(\frac{a h}{q}\right) \sum_{\substack{\sqrt{N}<p \le N \\
p\equiv h~(\bmod q)}} e\left(p \lambda\right)+O(\sqrt{N}) \\
= & \mu(q) \int_{\sqrt{N}}^{N}e(u\lambda)\mathrm{d}\sum_{p\le u, ~p\equiv h~({\rm mod}~q)}1+O(\sqrt{N}).
\end{align*}
When $q\le\log^{A}u$, by Siegel--Walfisz theorem, we have
$$
\sum_{p\le u, ~p\equiv h~({\rm mod}~q)}1=\frac{\mathrm{Li}(u)}{\varphi(q)}+O(ue^{-c\sqrt{\log u}}),  \quad\quad \mathrm{Li}(u)=\int_{2}^{u}\frac{\mathrm{d} t}{\log t}.
$$ 
Moreover, under the generalized Riemann hypothesis (GRH in brief), the error term can be improved to
$$
O(u^{1/2}\log^2 u).
$$
By substituting either of the above equations into the major arcs, and  using the  properties of characteristic functions with the integral in \cite[Theorem 11]{Hua-65}, we get the first equation.

The proof shows the error term in the asymptotic formula cannot be improved under GRH, which is different from Theorem \ref{<3>}. One needs to explore other methods for improving the error term.
\end{proof}
\vspace*{1ex}
\noindent\textbf{Data availability statement} Data sharing not applicable to this article as no datasets were generated or analysed during the current study.          

\vspace*{3 ex}
\noindent\textbf{\large Declaration}
\vspace*{3 ex}

\noindent\textbf{Conflict of interest} On behalf of all authors, the corresponding author states that there is no conflict of interest.

\bibliography{ref92}
\bibliographystyle{plain}
\end{document}